\documentclass[11pt,reqno,english]{amsart}
\usepackage[T1]{fontenc}
\usepackage[latin9]{inputenc}
\usepackage{geometry}

\usepackage{color}
\usepackage{babel}
\usepackage{mathrsfs}
\usepackage{mathtools}
\usepackage{bm}
\usepackage{amsbsy}
\usepackage{amstext}
\usepackage{amsthm}
\usepackage{amssymb}
\usepackage{graphicx}
\usepackage{setspace}
\usepackage{esint}

\usepackage{bbold}

\usepackage[unicode=true,pdfusetitle,
 bookmarks=true,bookmarksnumbered=false,bookmarksopen=false,
 breaklinks=false,pdfborder={0 0 1},backref=false,colorlinks=false]
 {hyperref}

\makeatletter

\numberwithin{equation}{section}

\newtheorem{theorem}{Theorem}[section]

\newtheorem{proposition}[theorem]{Proposition}
\newtheorem{corollary}[theorem]{Corollary}

\newtheorem{definition}[theorem]{Definition}
\newtheorem{remark}[theorem]{Remark}

\newtheorem{thm}[theorem]{Theorem}
\newtheorem{lem}[theorem]{Lemma}
\newtheorem{prop}[theorem]{Proposition}
\newtheorem{cor}[theorem]{Corollary}

\newtheorem{rem}[theorem]{Remark}

\usepackage{a4wide}

\newcommand{\mc}[1]{{\mathcal #1}}
\newcommand{\mf}[1]{{\mathfrak #1}}
\newcommand{\mb}[1]{{\boldsymbol #1}}
\newcommand{\bb}[1]{{\mathbb #1}}

\usepackage[dvipsnames]{xcolor}

\usepackage{cite}  

\usepackage{stmaryrd}
\usepackage{bbm}

\newcommand{\<}{\langle}
\renewcommand{\>}{\rangle}

\makeatletter
\providecommand{\leftsquigarrow}{%
  \mathrel{\mathpalette\reflect@squig\relax}%
}
\newcommand{\reflect@squig}[2]{%
  \reflectbox{$\m@th#1\rightsquigarrow$}%
}
\makeatother

\title[Driven gradient exclusion processes] {Dynamical large
deviations for boundary driven gradient symmetric exclusion processes
in mild contact with reservoirs}

\author[A. Bouley]{Ang\`ele Bouley} 
\address{Ang\`ele Bouley
  \hfill\break\indent CNRS UMR 6085, Universit\'e de
  Rouen, \hfill\break\indent Avenue de l'Universit\'e, BP.12,
  Technop\^ole du Madril\-let, \hfill\break\indent
F76801 Saint-\'Etienne-du-Rouvray, France.} 
\email{angele.bouley@univ-rouen.fr}

\author[C. Landim]{Claudio Landim} 
\address{Claudio Landim
  \hfill\break\indent IMPA \hfill\break\indent Estrada Dona Castorina
  110, \hfill\break\indent
J. Botanico, 22460 Rio de Janeiro, Brazil\hfill\break\indent
  {\normalfont and} \hfill\break\indent CNRS UMR 6085, Universit\'e de
  Rouen, \hfill\break\indent Avenue de l'Universit\'e, BP.12,
  Technop\^ole du Madril\-let, \hfill\break\indent
F76801 Saint-\'Etienne-du-Rouvray, France.} 
\email{landim@impa.br}

\begin{document}

\begin{abstract}
We consider a one-dimensional gradient symmetric exclusion process in
mild contact with boundary reservoirs. The hydrodynamic limit of the
empirical measure is given by a non-linear second-order parabolic
equation with non-linear Robin boundary conditions.  We prove the
dynamical large deviations principle.
\end{abstract}

\maketitle

\section{Introduction}

A major challenge addressed in recent years consists in the physical
description of non-equilibrium states (cf. \cite{gm, mft} and references
therein).

The analysis of interacting particle systems stationary states, as the
number of degrees of freedom diverges, has proved to be an important
step in this program. The thermodynamical functional of the empirical
measure, the observable of interest in lattice gases, corresponds to
the large deviation rate functional for the empirical measure when the
particles are distributed according to the stationary state.

In equilibrium, when there is no current of particles, the invariant
measure is given by the Gibbs distribution specified by the
Hamiltonian, and the thermodynamical functional is called the free
energy.  In non-equilibrium, in the presence of currents, the
invariant measure is usually not known explicitly, and a basic problem
consists in characterizing the rate functional, known as the
non-equilibrium free energy.

In certain special cases \cite{dls, dhs}, combinatorial arguments
based on special feature of the exclusion processes provide a formula
for the non-equilibrium free energy. For general dynamics, a dynamical
approach has been proposed \cite{bdgjl2}.  One fist prove a dynamical
large deviations principle for the empirical measure in a fixed time
interval $[0,T]$. The non-equilibrium free energy of a density profile
$\gamma$ is given by the minimal cost among all trajectories which
start from the stationary density profile and reaches $\gamma$ in
finite time. This procedures provides a dynamical variational formula
for the non-equilibrium free energy.

In this article we continue the investigation started in \cite{bl,
fgln, bel} of the dynamical large deviations for interacting particle
systems whose hydrodynamic equation is described by a second-order
parabolic equation with Robin boundary conditions. We consider a
gradient exclusion process leading to a non-linear hydrodynamic
equation.  This non-linearity turns the large deviations rate
functional non-convex, a property much used in the proof of the large
deviations principle lower-bound in the linear case \cite{fgln}.

In possession of the dynamical large deviations, one can proceed as in
\cite{bg, f, bdgjl3, bel} to prove the large deviations for the
empirical measure under the invariant measure, deriving a dynamical
variational formula for the non-equilibrium free energy. This is left
for a future work. The next step, a challenging open problem, consists
in obtaining a simple formula for the non-equilibrium free energy,
derived in \cite{dls, bdgjl2} for symmetric exclusion processis in
strong contact with reservoirs and in \cite{dhs, bel} for these
dynamics in mild contact with reservoirs.

This article presents two novelties which can be useful in other
contexts, as for reaction-diffusion models without the concavity
assumption on the creation and destruction rates \cite{JLV93, BL12,
FLT19}.

Denote by $I_{[0,T]}(\cdot)$ the large deviations rate functional.
One of the main difficulties in the proof of the lower bound consists
in showing that any trajectory $\pi (t,dx) = u(t,x)\, dx$ with finite
rate function can be approximated by a sequence of smooth paths
$\pi_n (t,dx) = u_n(t,x)\, dx$ such that $\pi_n \to \pi$,
$I_{[0,T]}(\pi_n) \to I_{[0,T]}(\pi)$.

The proof of this density is divided in several steps. When the rate
functional is not convex, the standard argument to settle one of them
\cite{qrv, blm} relies on a Riesz representation lemma of linear
functionals in negative Sobolev spaces (cf. \cite[Lemma 4.8]{blm}). We
present here a much simpler and direct proof of this step.

The second technical novelty concerns the regularity of paths with
finite rate functionals.  Fix a path $\pi (t,dx) = u(t,x)\, dx$ such
that $I_{[0,T]}(\pi) < \infty$.  It is known, at least since
\cite{blm}, that such a path is weakly continuous in time. More
precisely, that for each space-dependent $C^1$ function $g$, the map
$t\mapsto \<\pi_t,g\>$ is uniformly continuous. Proposition \ref{th:
IPP} below asserts that this map is absolutely continuous and its
derivative belongs to $L^1$. This property permits to integrate by
parts $\int_0^T \< \partial_t H_t , u_t\>\, dt$, when
$H(t,x) = h(t)\, g(x)$ and plays a crucial role in the decomposition
of the rate functional presented in Section \ref{sec5} below.

\section{Notation and results}
\label{sec2}

\subsection*{The model}

We consider a one-dimensional boundary driven gradient symmetric
exclusion processes in weak contact with boundary reservoirs.  Fix
$N\ge 1$, and let $\color{blue} \Lambda_N=\{-N+1 ,\dots, N-1\}$.  The
state-space is represented by
$\color{blue} \mf S_N =\{0,1\}^{\Lambda_N}$ and the configurations by
the Greek letters $\eta$, $\xi$ so that $\color{blue} \eta (x)$,
$x\in \Lambda_N$, represents the number of particles at site $x$ for
the configuration $\eta$. Here and below all notation introduced in
the text and not in displayed equations is indicated in blue.

The theorems stated in this article apply to all gradient symmetric
exclusion dynamics in mild contact with reservoirs. To avoid heavy
notation, we consider a specific example.  Fix {\color{blue}
$a> -1/2$, $0<\mf a$, $\mf b<1$}. The variables $\mf a$, $\mf b$
represent the density of the left, right reservoirs,
respectively. Assume, without loss of generality, that
$\mf a\le \mf b$. Let $r_{x,x+1} \colon \mf S_N \to \bb R_+$,
$-N+2\le x\le N-3$, be the rates at which the variables $\eta(x)$,
$\eta(x+1)$ are exchanged, and given by
\begin{align*}
r_{x,x+1}(\eta)=1+a \,[\, \eta(x-1)+\eta(x+2) \,] \;.
\end{align*}
At the boundary, replace the occupation variables which do not belong
to $\Lambda_N$ by the density of the reservoirs:
\begin{gather*}
r_{-N+1,-N+2}(\eta)=1+a\, [\, \mf a + \eta(-N+3)\,] \;, \\
r_{N-2,N-1}(\eta)=1+a\, [\, \eta(N-3)+{\mf b}\,]\;.
\end{gather*}
The generator $L_N$ of the Markov process is given by
\begin{align*}
L_N=L_{N,0}+L_{N,b}
\end{align*}
where $L_{N,0}$, $L_{N,b}$ correspond to the bulk, boundary dynamics,
respectively. The action of the generator $L_{N,0}$ on functions
$f\colon \mf S_N \to \mathbb{R}$ is given by
\begin{align*}
(L_{N,0}f)(\eta)=N^2\sum_{x }r_{x,x+1}(\eta)
\, [\, f(\eta^{x,x+1})-f(\eta)\,] \;,
\end{align*}
where the sum is carried over all $x\in \bb Z$ such that
$\{x , x+1\} \subset \Lambda_N$. Moreover, $\eta^{x,x+1}$ stands for
the configuration obtained from $\eta$ by exchanging the occupation
variables $\eta(x)$ and $\eta(x+1)$:
\begin{align*}
\eta^{x,x+1}(z)=
\begin{cases}
\eta (z) &\mbox{ if } z \neq x,x+1\\
\eta (x+1)&\mbox{ if } z=x \\
\eta (x)  &\mbox{ if } z=x+1\;.
\end{cases}
\end{align*}

The generator of the boundary dynamics acts on functions
$f\colon \mf S_N \to \mathbb{R}$ as
\begin{align*}
(L_{N,b}f)(\eta)= N \,\Big\{ \,
r_L (\eta)\, [\, f(\eta^{-N+1})-f(\eta)\, ] \,+\,
r_R (\eta)\, [\, f(\eta^{N-1})-f(\eta)\, ] \, \Big\}\;.
\end{align*}
In this formula, $\eta^x$ represents the configuration obtained from
$\eta$ by flipping the occupation variable at site $x$:
\begin{align*}
\eta^x(z)=
\begin{cases}
\eta (z) &\mbox{ if } z \neq x\;,\\
1-\eta (x)&\mbox{ if } z=x \;,\\
\end{cases}
\end{align*}
and the rates $r_L$, $r_R$ are given by
\begin{gather*}
r_L(\eta)=\eta(-N+1)\, (\,1-{\mf a}\,) \,+\,
[\,1-\eta(-N+1)\, ]\, {\mf a}\;, \\
r_R(\eta)=\eta(N-1)\, (\,1-{\mf b}\,) \,+\,
[\,1-\eta(N-1)\,]\, {\mf b} \;.
\end{gather*}

Note that the bulk dynamics has been accelerated by $N^2$, while the
boundary one by $N$.  Denote by $(\eta^N_t)_{t \geq 0}$ the Markov
process associated to the generator $L_N$. For a smooth function
$\rho\colon \Omega \to ]0,1[$, let $\nu_{\rho(\cdot)}^N$ be the
Bernoulli product measure on $\mf S_N$ with marginals given by:
\begin{align*}
\nu_{\rho(\cdot )}^N(\eta(x)=1)=\rho ( x/N )   \;, \quad
x\in \Lambda_N\;.
\end{align*}
When the function $\rho$ is constant equal to $\mf c\in [0,1]$, we
denote $\nu_{\rho(\cdot )}^N$ by $\color{blue} \nu_{\mf c}^N$.  An
elementary computation shows that $\nu_{\mf a}^N$ is a reversible
stationary state for the Markov chain $\eta^N_t$ when $\mf a = \mf b$.

For a metric space $\mathcal{E}$ and $T>0$, denote by
$\color{blue} D([0,T],\mathcal{E})$, the space of right-continuous
functions with left limits $\mb x\colon [0,T] \to \mc E$ endowed with
the Skorohod topology and its Borel $\sigma$-algebra.  For a
probability measure $\mu$ on $\mf S_N$, let
$\color{blue} \mathbb{P}^N_{\mu}$ be the measure on
$D([0,T], \mf S_N)$ induced by the continuous-time Markov process
associated to the generator $L_N$ starting from $\mu$. Expectation
with respect to $\mathbb{P}^N_{\mu}$ is denoted by
$\color{blue} \mathbb{E}^N_{\mu}$.

\subsection*{Gradient condition}

For $x\in \bb Z$ such that $\{x,x+1\} \subset \Lambda_N$, denote by
$j_{x,x+1}$ the instantaneous current over the bond $(x,x+1)$. This is
the rate at which a particle jumps from $x$ to $x+1$ minus the rate at
which it jumps from $x+1$ to $x$:
\begin{equation*}
j_{x,x+1} (\eta) = r_{x,x+1} (\eta)\, [\, \eta(x) - \eta(x+1)\,] \;.
\end{equation*}
By definition of the rates $r_{x,x+1}$, the current is given by
\begin{equation}
\label{14}
j_{x,x+1} (\eta)  =  [\, \eta(x) - \eta(x+1)\,] 
\,+\, a\, (\, \tau_{x-1} f_1 - \tau_{x+1} f_1\,)
\,+\, a\, (\, \tau_{x} f_2 - \tau_{x-1} f_2\,) \;,
\end{equation}
where $\{\tau_x: x\in \bb Z\}$ stands for the group of translations,
$(\tau_x \eta)(y) = \eta(x+y)$, $y\in\bb Z$, and
$(\tau_x f)(\eta) = f(\tau_x \eta)$,
\begin{equation*}
f_1 (\eta) \,=\, \eta(0) \, \eta(1)  \;, \quad
f_2 (\eta) \,=\, \eta(0) \, \eta(2)  \;.
\end{equation*}
Note that the instantaneous current can be written as a sum of
functions minus their translations. This property is called the
gradient condition, cf. \cite[Chapter 7]{kl}. We carefully defined the
jump rates $r_{x,x+1}(\eta)$ in order to fullfil this condition.

Denote by $\color{blue} \Omega$ the open set $]-1,1[$.  Let
$\color{blue} \mathcal{M}$ be the set of non-negative measures on
$\Omega$ with total mass bounded by 2 endowed with the weak topology.
For a continuous function $F\colon [-1,1] \to \mathbb{R}$ and a
measure $\pi \in \mathcal{M}$, $\<\pi, F\>$ represents the integral of
$F$ with respect to $\pi$:
\begin{align*}
{\color{blue} \<\pi,F\>}
=\int_{\Omega} F(x)\, \pi(dx)\;.
\end{align*}

Denote by $\pi^N=\pi^N(\eta)$, $\eta\in \mf S_N$, the empirical
measure obtained by assigning mass $1/N$ to each particle of a
configuration $\eta$:
\begin{equation}
\label{13}
{\color{blue} {\pi^N}}
\,=\, \frac{1}{N}\sum_{x \in \Lambda_N} \eta(x) \,
\delta_{x/N} \;\; \text{and let}\;\;
{\color{blue} \pi^N_t } = \pi^N(\eta^N_t) 
\,:=\, \frac{1}{N}\sum_{x \in \Lambda_N} \eta^N_t(x) \,
\delta_{x/N} \;,
\end{equation}
where $\delta_y$ represents the Dirac mass at $y$.

For an open subset $\mc A$ of $\mathbb{R}$, $\color{blue} C^m(\mc A)$,
$1 \leq m \leq + \infty$, stands for the space of real functions
defined on $\mc A$ that are $m$ times continuously differentiable. Let
$\color{blue} C^m_0(\mc A)$, $\color{blue} C_K^m(\mc A)$,
$1 \leq m \leq + \infty$, be the subset of functions in $C^m(\mc A)$
which vanish at the boundary of $\mc A$, with compact support in
$\mc A$, respectively.

Fix a function $G$ in $C^3(\Omega)$, which can be extended
continuously, as well as its derivatives to $\overline{\Omega}$.  By
the gradient condition \eqref{14}, which permits two summation by
parts and, in consequence, a Taylor expansion of $G$ up to the second
order, providing in this way a factor $1/N^2$,
\begin{align*}
L_N \< \pi^N , G\> \;=\;
\frac{1}{N}\sum_{x=-N+2}^{N-2} (\Delta G)(x/N)\, \big\{\,
\eta(x) + 2\, a\, \tau_x f_1 \,-\, a\, \tau_{x} f_2\,\big\}
\;+\; B_N\; +\; R_N \;,
\end{align*}
where $B_N=B^L_N + B^R_N$ is the boundary term
\begin{align*}
B^L_N \; & =\; G(-1) \,[\,\mf a - \eta(-N+1) \,] \;+\;
(\nabla G)(-1) \, \eta(-N+2) \\
& +\, a\, (\nabla G)(-1) \, \big\{\, 
\tau_{-N+1} \, f_1 \,+\,  \tau_{-N+2}  \, f_1
\,-\,  \tau_{-N+1} \, f_2\,\big\} \\
& -\, (\nabla G)(-1) \, \big[\, \eta(-N+2) -\eta(-N+1)\, \big]\,
\big\{\, 1 + a\, \mf a +  a \eta(-N+3) \,\big\}\;,
\end{align*}
\begin{align*}
B^R_N \; & =\; G(1) \,[\,\mf b - \eta(N-1) \,] \;-\;
(\nabla G)(1) \, \eta(N-2) \\
& -\, a (\nabla G)(1) \, \big\{\, 
\tau_{N-2} \, f_1 \,+\,  \tau_{N-3}  \, f_1
\,-\,  \tau_{N-3} \, f_2\,\big\} \\
& -\, (\nabla G)(1) \, \big[\, \eta(N-1) -\eta(N-2)\, \big]\,
\big\{\, 1 + a\, \mf b + a \eta(N-3) \,\big\}\;,
\end{align*}
and $R_N$ is a remainder whose absolute value is bounded by $C_0(G)/N$
for some finite constant $C_0$ which depends on $G$ but not on
$N$. All constants below are allowed to depend on the parameters $a$,
$\mf a$, $\mf b$ with any mention to that in the notation.  In the
previous formula, $\color{blue} \Delta G$, $\color{blue} \nabla G$
stand for the Laplacian, gradient of $G$, respectively.

\subsection*{Heuristics}

Consider the martingale $M_N(t)$ defined by
\begin{align*}
M_N(t) \;=\; \<\pi^N_t, G\> \,-\, \<\pi^N_0, G\> \,-\,
\int_0^t L_N \<\pi^N_s, G\> \, ds\;.
\end{align*}
where $G\colon [-1,1]\to \bb R$ is a smooth function.  A computation
of its quadratic variation, similar to the one performed in
\cite[Chapter 4]{kl} yields that it vanishes in $L^2$ as $N\to\infty$.

For a function $f\colon\{0,1\}^{\bb Z} \to \bb R$ which depends only
on a finite number of coordinates, let $\widehat f\colon [0,1]\to \bb R$
be given by
\begin{align*}
\widehat f (\mf c) \,=\, E_{\nu^N_{\mf c}}[\,f\,]\;.
\end{align*}
In this formula, $N$ provided that $N$ has been chosen sufficiently
large to contain the support of $f$.

If one is allowed to replace $(\tau_x f) (\eta^N_s)$ by $\widehat
f (u(s, x/N))$ for some measurable function $u\colon [0,T]\times
[-1,1]\to [0,1]$, which is essentially the content of the so-called
two blocks estimate, it follows from the convergence of the martingale
$M_N(t)$ to $0$ and from the computation of the previous subsection that
\begin{align*}
& \int_{-1}^1 G(x)\, u(t,x)\, dx \;-\;
\int_{-1}^1 G(x)\, u(0,x)\, dx  \,-\,
\int_0^t \int_{-1}^1 (\Delta G) (x)\,
P_a (u(s,x)) \, dx \, ds \\
&\quad =\; \int_0^t \Big\{\, G(-1) \,  [\, \mf a - u(s,-1)\,] \,+\, 
(\nabla G)(-1) \, P_a(u (s,-1)) \,\Big\}\, ds \\
&\quad -\; \int_0^t \Big\{\, G(1) \,  [\, \mf b - u(s,1)\,] \,+\, 
(\nabla G)(1) \, P_a(u (s,1)) \,\Big\}\, ds \, + R_N\;, 
\end{align*}
where $P_a(z) = z + a^2 z$. The last term in $B^L_N$, $B^R_N$ does not
contribute because it has mean-zero with respect to all measures
$\nu^N_{\mf c}$.

Integrating twice by parts yields that $u(t,x)$ is the solution of the
partial differential equation
\begin{equation}
\label{10}
\left\{
\begin{aligned}
& \partial_t u \;=\; \nabla \big( D(u) \nabla u\,) \;, \quad
(t,x) \in (0,\infty) \times \Omega\;, \\
& D(u) \nabla u =  u - \mf a  \;, \quad 
(t,x) \in (0,\infty) \times \{-1\} \;, \\
&D(u) \nabla u =  \mf b - u  \;, \quad 
(t,x) \in (0,\infty) \times \{1\}\;,
\end{aligned}
\right.
\end{equation}
where
\begin{equation*}
{\color{blue} D(\rho) = P'_a(\rho) = 1 + 2a\rho}\;.
\end{equation*}

\subsection*{Hydrodynamic limit}

Let $\Pi_N \colon D([0,T], \mf S_N) \to D([0,T], \mathcal{M})$ be the
map which associates to a trajectory $(\eta^N_t : 0\le t\le T)$ the
empirical measure trajectory $(\pi^N_t : 0\le t\le T)$.  For a
probability measure $\mu$ in $\mf S_N$, $\mathbb{Q}_{\mu}^N$ stands
for the measure on $D([0,T], \mathcal{M})$ given by
$\color{blue} \mathbb{Q}_{\mu}^N=\mathbb{P}_{\mu}^N\, (\Pi^N)^{-1}$.

A sequence of measures $(\mu^N)_{N \in \mathbb{N}}$ on $\mf S_N$ is said
to be associated to a density profile $\rho_0\colon \Omega \to [0,1]$
if
\begin{equation}
\label{12}
\underset{N \to \infty}{ \lim} \mu_N \Big [ \,
\Big | \, \<\pi_N \,,\, G \> \,-\,
\int_{\Omega}G(x)\, \rho_0(x)\, dx \,\Big | >\delta \,\Big ] 
\,=\, 0 
\end{equation}
for all continuous functions $G\colon [-1, 1] \to \mathbb{R}$ and
$\delta > 0$.  The first main result of this article establishes the
hydrodynamic limit of the empirical measure.

Denote by $\color{blue} C^{m,n}([0,T]\times [-1,1])$,
$1 \leq n,m \leq + \infty$, the space of functions
$G\colon [0,T]\times [-1,1] \to \mathbb{R}$ with $m$ continuous
derivatives in time and $n$ continuous derivatives in space.

Let $\color{blue} \mc H^1 = \mc H^1(\Omega)$ be the Sobolev space of
measurable functions $G\colon \Omega \to \mathbb R$ with generalized
derivatives $\nabla G$ in $L^{2}(\Omega)$. The space $\mc H^1$ endowed
with the scalar product $\langle\cdot,\cdot\rangle_{1}$, defined by
\begin{equation}
\label{n08}
\langle G \,,\, H \rangle_{1}\;:=\;
\langle G \,,\, H\rangle \;+\;
\langle \nabla  G \,,\, \nabla  H\rangle\;,
\end{equation}
is a Hilbert space. The corresponding norm is denoted by
$\|\cdot\|_{\mc H^1}$:
\begin{equation*}
\|G\|_{\mc H^1}^{2} \;:=\;
\int_{-1}^1 |G(x)|^2 \;dx \;+\;
\int_{-1}^1 |\nabla  G(x)|^2 \;dx\;.
\end{equation*}
Recall from \cite{fgln} that any function $H$ in $\mc H^1$ has a
continuous version.

\begin{definition}
\label{d01}
Fix $T>0$ and a density profile $\rho_0\colon [-1,1] \to [0,1]$. A
mesurable function $u\colon [0,T] \times [-1,1] \to [0,1]$ is said to
be a weak solution of the differential equation \eqref{10} with
initial condition $u(0,\cdot) = \rho_0(\cdot)$ if
\begin{enumerate}
\item[(a)] $u$ belongs to $L^2([0,T],\mathcal{H}^1)$:
\begin{equation*}
\int_0^T \left ( \int_{-1}^1 | \nabla u(s,x) |^2dx\right )ds < \infty
\end{equation*}

\item[(b)] For every function $H$ in $C^{1,2}([0,T] \times [-1,1]$:
\begin{align*}
\<u_T,H_T\> \,-\, \<u_0,H_0\>
&=\int_0^T\<u_s, \partial_s H_s\> \, ds
\,-\, \int_0^T \<D(u_s) \, \nabla u_s, \nabla H_s\>\, ds \\
& + \int_0^T [\, {\mf a}-\rho_s(-1)\,]\, H_s(-1)\, ds 
+\int_0^T [\, {\mf b}-\rho_s(1)\,] \, H_s(1) \, ds \;.
\end{align*}
\end{enumerate}
\end{definition}

Corollary \ref{t03} asserts that there is at most one weak solution of
\eqref{10}. Existence for any initial condition $\rho_0\colon \Omega
\to [0,1]$ follows from the tightness part of the next theorem. 

\begin{thm}
\label{t01}
Fix $T>0$ and a profile $\rho_0\colon \Omega \to [0,1]$. Let
$(\mu^N)_{N \in \mathbb{N}}$ be a sequence of measure on $\mf S_N$
associated to $\rho_0$. Then, the sequence of probability measures
$(\mathbb{Q}_{\mu^N}^N)_{N \in \mathbb{N}}$ converges weakly to the
probability measure $\mathbb{Q}$ concentrated on the trajectory
$\pi(t,dx)=u(t,x)\, dx$, where $u$ is the unique weak solution of the
hydrodynamic equation \eqref{10} with initial condition
$\rho_0(\cdot)$.
\end{thm}

The proof of this theorem is given in Section \ref{sec3}.

\subsection*{The rate functional}

Fix $T>0$. Let $\mathcal{M}^0$ be the subset of $\mathcal{M}$ of all
absolutely continuous measures with respect to the Lebesgue measure
with positive density bounded by 1:
\begin{align*}
{\color{blue} \mathcal{M}^0} \,:=\,  
\{ \pi \in \mathcal{M} 
\mbox{ : } \pi(dx)= \gamma (x)dx \mbox{ and } 
0\leq \gamma (x) \leq 1 \mbox{ a.e. }\}  \;.  
\end{align*}

\begin{definition}
Denote by $\mathcal{Q}\colon D([0,T],\mathcal{M}^0) \to \bb R_+$ the
energy given by
\begin{align*}
\mathcal{Q}(\pi) \,=\, \underset{G}{\sup}\, 
\Big \{ \, 2\int_0^T \<u_t,\nabla G_t\> \, dt 
- \int_0^T \int_{-1}^1 \chi (u(t,x))\, G(t,x)^2\, dx\, dt \, \Big \}\;, 
\end{align*}
where $\pi(t,dx) = u(t,x)\, dx$, the supremum is carried over all
functions $G \in C^{1,2}([0,T] \times [-1,1])$ and
$\color{blue} \chi(a) = a(1-a)$ is the static compressibility of the
exclusion process.
\end{definition}

Let $\color{blue} D_{\mc E} ([0,T], \mc M^{0})$ be the trajectories in
$D ([0,T], \mc M^{0})$ with finite energy.  Fix $0\le r< s \le T$, and
a trajectory $\pi \in D_{\mc E}([0,T],\mathcal{M}^0)$,
$\pi(t,du)=u(t,x)dx$.  Denote by
$L_{[r,s]} \colon C^{1,1}([0,T]\times [-1,1]) \to \mathbb{R}$ the
linear functional given by
\begin{equation}
\label{07}
{\color{blue} L_{[r,s]}(H)} \,=\, \<\pi_s, H_s\> \,-\, \<\pi_r, H_r\>
\,-\, \int_r^s\<\pi_t, \partial_t H_t\> \, dt \;.
\end{equation}
When $r=0$, $s=T$, we represent $ L_{[s,t]}$ by $L_0$:
$\color{blue} L_{0}(H) = L_{[0,T]}(H)$.

For each $H \in C^{1,1}([r,s] \times [-1,1])$, let
$\hat{J}_H=\hat{J}_{H,[r,s]} \colon D_{\mc E}([0,T],\mathcal{M}^0) \to
\bb R$ be the functional given by:
\begin{align*}
\hat{J}_H(\pi) \, & =\, L_{[r,s]}(H)
\,+\, \int_r^s \int_{-1}^1 D(u(t,x))\, \nabla u(t,x)
\, \nabla H(t,x) \,dx\, dt \\
& - \, \int_r^s \int_{-1}^1 \sigma(u(t,x))\,
(\nabla H(t,x))^2 \,dx \, dt \,-\
\int_r^s B(u_t,H_t)\, dt \;, 
\end{align*}
where $\color{blue} \sigma (a)= \chi(a)\, D(a)$ is the mobility, and 
\begin{equation}
\label{11}
\begin{aligned}
& {\color{blue} B(u_t,H_t)}
\, :=\, u_t(1)\, (1-{\mf b})\,
\left ( e^{-H(t,1)}-1\right )\, +\,
{\mf b}\, [1-u_t(1)]\,  \left ( e^{H(t,1)}-1\right )\\
& \;\; +u_t(-1)\,  (1-{\mf a})\, \left ( e^{-H(t,-1)}-1\right )
\, +\, {\mf a}\, [1-u_t(-1)]\,  \left ( e^{H(t,-1)}-1\right ) \;.
\end{aligned}
\end{equation}
We extend the definition of the functional $\hat J_H$ to the space
$D([0,T],\mathcal{M})$ by setting
\begin{equation*}
{\color{blue} J_H(\pi) } \;=\; 
\left \{ 
\begin{array}{ll}
\hat{J}_H(\pi) & \mbox{if } \pi 
\in D_{\mc E} ([0,T],\mathcal{M}^0) \;, \\
+ \infty & \mbox{otherwise }\;.  
\end{array}
\right .
\end{equation*}

Fix a profile $\gamma\colon \Omega \to [0,1]$ and $0\le r<s\le
T$. Denote by $I_{[r,s]} \colon  D([0,T],\mathcal{M}) \to [0,+\infty]$
the functional given by
\begin{align*}
{\color{blue} I_{[r,s]}(\pi)} \,:=\, 
\underset{H \in C^{1,1}([r,s] \times [-1,1])}{ \sup }  \,
J_H(\pi) \;,
\end{align*}
and let  
\begin{equation*}
{\color{blue} I_{[r,s]}(\pi|\gamma)} \;=\; 
\left \{ 
\begin{array}{ll}
I_{[r,s]}(\pi)  & \mbox{if } \pi_0 (dx)  \,=\, \gamma(x)\, dx \;, \\
+ \infty & \mbox{otherwise }\;.  
\end{array}
\right .
\end{equation*}

\begin{theorem}
\label{mt2}
Fix $T>0$ and a measurable function $\gamma \colon [-1,1] \to [0,1]$.
The functional
$I_{[0,T]}(\cdot|\gamma) \colon D([0,T],\mathcal{M})\to[0,\infty]$ is
lower semicontinuous and has compact level sets.
\end{theorem}

This result is proved in Section \ref{sec4}, where we present some
properties of the rate functional $I_{[0,T]}(\cdot)$.

\subsection*{Dynamical large deviations}

For $0<\beta<1$, denote by $\color{blue} C^{n+\beta}([-1,1])$, $n$ a
non-negative integer, the functions in $C^{n}([-1,1])$ whose $n$-th
derivative is H\"older continuous with parameter $\beta$.  The main
theorem of this article reads as follows.

\begin{thm}
\label{t02}
Fix $T>0$, a density profile $\gamma \in C^{2+\beta}([-1,1])$ for some
$0<\beta<1$ satisfying the boundary conditions
\begin{equation}
\label{fv-01}
D(\gamma (-1))  \, \nabla \gamma  (-1) \,=\, \gamma(-1) - \mf a\;, \quad
D(\gamma(1)) \, \nabla \gamma (1) \,=\, \mf b - \gamma(1) \;,
\end{equation}
and a sequence of configurations $\{\eta^N\}_{N\in \bb N}$.  Assume
that $\delta_{\eta^N}$ is associated to $\gamma$ in the sense of
\eqref{12}.  Then, the sequence of probability measures
$\{\bb Q_{\eta^N}\}_{N\geq 1}$ satisfies a large deviation principle
with speed $N$ and good rate function $I_{[0,T]}(\cdot|\gamma)$.  Namely,
for each closed set $C \subset D([0,T],\mathcal{M})$,
\begin{align*}
\limsup_{N \to \infty}
\frac{1}{N}\log \mathbb{Q}_{\eta^N}(C) \leq
-\inf_{\pi \in C}  I_{[0,T]}(\pi |\gamma)
\end{align*}
and for each open set $O \subset D([0,T],\mathcal{M})$,
\begin{align*}
\liminf_{N\to\infty}  
\frac{1}{N} \log \mathbb{Q}_{\eta^N}(O) 
\geq - \, \inf_{\pi \in O}  I_{[0,T]}(\pi |\gamma)\;.
\end{align*}
\end{thm}

We prove this theorem in Section \ref{sec7}.

\begin{remark}
\label{rm3}
The hypotheses that $\gamma$ belongs to $C^{2+\beta}([-1,1])$ for some
$0<\beta<1$ and satisfies the boundary conditions \eqref{fv-01} are
only used in the proof of Proposition \ref{p07} where we derive a
maximum principle for the solutions of the hydrodynamic equation
\eqref{10}.

In other words, let $u$ be the solution of the hydrodynamic equation
\eqref{10} with initial condition $\gamma$.  If one is able to prove,
under weaker assumptions on $\gamma$, that for each $\delta >0$, there
exists $\epsilon >0$ such that $\epsilon \le u (t,x) \le 1-\epsilon$
for every $(t,x) \in [\delta,T] \times [-1,1]$, then Theorem \ref{t02}
holds under these weaker hypotheses.
\end{remark}

\subsection*{$I_{[0,T]}$-density}

The main step in the proof of the large deviations lower bound consists
in showing that any trajectory $\pi$ in $D([0,T],\mc {M})$ with finite
rate function can be approximated by a sequence of regular
trajectories $\pi^n$ in such a way that $I_{[0,T]}( \pi^{n}|\gamma)$
converges to $I_{[0,T]}( \pi|\gamma)$. The precise statement requires
some notation.

\begin{definition}
\label{d05}
A subset $A $ of $ D([0,T], \mathcal{M})$ is said
$I_{[0,T]}(.|\gamma)$-dense if for any $\pi \in D([0,T],\mathcal{M})$
such that $I_{[0,T]}(\pi|\gamma)< \infty$, there exists a sequence
$(\pi^n)_{n \in \mathbb{N}}$ in $A$ such that $\pi^n \to \pi$ in
$D([0,T],\mathcal{M})$, and
$I_{[0,T]}(\pi^n|\gamma) \to I_{[0,T]}(\pi|\gamma)$.
\end{definition}

\begin{definition}
\label{d04}
Given $\gamma\colon [-1,1]\to [0,1]$, let $\Pi_\gamma$ be the
collection of all paths $\pi(t,dx) = u(t,x) dx$ in
$D([0,T], \mc M^0)$ such that
\begin{itemize}
\item[(a)] There exists $\mf t >0$, such that $u$ follows the
hydrodynamic equation \eqref{10} with initial condition
$\gamma(\cdot)$ in the time interval $[0, \mf t]$. In particular,
$u(0,\cdot) = \gamma (\cdot)$.

\item[(b)] There exists $\epsilon>0$ such that
$\epsilon \le u(t,x) \le 1-\epsilon$ for all $(t,x)$ in
$[\mf t , T]\times[-1,1]$, and $u$ is smooth on
$(\mf t,T]\times [-1,1]$.
\end{itemize}
\end{definition}

Mind that in this definition we do not assume that $\gamma$ fulfils
the conditions of Theorem \ref{t02}.  Proposition \ref{p05} states
that for density profiles $\gamma: [-1,1] \to [0,1]$ satisfying the
hypotheses of Theorem \ref{t02}, the set $\Pi_\gamma$ is
$I_{[0,T]}(\cdot|\gamma)$-dense. Moreover, under the hypotheses of
Theorem \ref{t02}, the first requirement in condition (b) can be
extended to the set $(0,T]\times [-1,1]$. Indeed, by Proposition
\ref{p07}, for every $0<\delta \le T$, there exists $\epsilon>0$ such
that $\epsilon \le u(t,x) \le 1-\epsilon$ for all $(t,x)$ in
$[\delta , T]\times[0,1]$.

At the end of Section \ref{sec6}, we provide an explicit formula for
the rate function of trajectories in $\Pi_\gamma$. This result does
not require the density profile $\gamma$ to satisfy the hypotheses of
Theorem \ref{t02}. For $0\le \varrho \le 1$, let $\mf p_{\varrho}$,
$\mf c_{\varrho} \colon [0,1] \times T \to \bb R$ be given by
\begin{gather}
\label{5-02}
\mf p_{\varrho} (a,M) \;=\; [1-a]\, \varrho \,
e^M \;-\; a\, [1-\varrho] \, e^{-M}    \;,
\\
\mf c_{\varrho} (a,M) \;=\;  [1-a]\, \varrho \,
[1 - e^M + M e^M ] \;+\; a\, [1-\varrho] \,
[\, 1 - e^{-M}  - M  e^{-M}]  \;. \nonumber
\end{gather}

\begin{proposition}
\label{l09b}
Fix a density profile $\gamma: [0,1]\to [-1,1]$ and a trajectory $\pi$
in $\Pi_\gamma$. Then, for each $t>0$, the elliptic equation (for $H$)
\begin{equation}
\label{5-01b}
\left\{
\begin{aligned}
& \partial_t u \;=\; \nabla (D(u) \nabla u)  \,-\,
2\, \nabla  \{ \sigma(u) \ \nabla  H \}\; , \\
& [D(u) \nabla  u]  (t,1) \,-\, 2\, \sigma(u(t,1)) \, \nabla  H(t,1) \,=\,
\mf p_{\mf b} \big(\, u(t,1)\,,\, H(t,1)\, \big) \;, \\
& [D(u) \nabla  u] (t,-1) \,-\, 2\, \sigma(u(t,-1)) \, \nabla  H(t,-1) \,=\,
-\, \mf p_{\mf a} \big(\, u(t,-1)\,,\, H(t,-1)\, \big) \;,
\end{aligned}
\right.
\end{equation}
has a unique solution, denoted by $H_t$. The function $H$ belongs to
$C^{1,1}(]\mf t,T]\times [-1,1])$, and the rate functional $I_{[0,T]} (u)$
takes the form
\begin{equation}
\label{5-03}
\begin{aligned}
I_{[0,T]} (u) \; & =\; \int_0^T dt \int_{0}^1 \sigma(u_t)\,
( \nabla  H_t)^2\, dx \; +\; \int_0^T
\mf c_{\mf b} \big(\, u_t(1)\,,\, H_t(1)\, \big)\; dt \\
\; & +\; \int_0^T
\mf c_{\mf a} \big(\, u_t(-1)\,,\, H_t(-1)\, \big)\; dt \;.
\end{aligned}
\end{equation}
\end{proposition}

Mind that $H_t=0$ for $0<t\le \mf t$ since $u$ follows the
hydrodynamic equation on the interval $[0,\mf t]$. By the construction
presented in Section \ref{sec6}, the time derivative of 
$H$ at $t=\mf t$ might be discontinuous.

\subsection*{Comments}

Since the first papers on the subject \cite{dv, kov}, robust methods
and tools have been developed to derive dynamical large deviations
principles for the empirical measures of diffusive interacting
particle systems \cite{kl, qrv, blm, flm, fgln}.  With these tools,
the proof of the large deviations principle boils down, essentially,
to the $I_{[0,T]}(\cdot|\gamma)$-density of a set of smooth
trajectories. 

The proof of the $I_{[0,T]}(\cdot|\gamma)$-density of the set
$\Pi_\gamma$, introduced in Definition \ref{d04}, presents two
novelties which might be helpful for other dynamics, as
reaction-diffusion models without the concavity assumption on the
creation and destruction rates \cite{JLV93, BL12, FLT19}.

Here, as in \cite{qrv, blm, flm}, the rate functional
$I_{[0,T]}(\cdot|\gamma)$ is not convex.  The lack of convexity
complicates the second step in the proof of the
$I_{[0,T]}(\cdot|\gamma)$-density, handled in Lemma \ref{lc02}. There,
we approximate a trajectory $\pi(t,dx) = u(t,x) dx$ by
$\pi^\epsilon(t,dx) = u^\epsilon(t,x) dx$, where
$u^\epsilon = (1-\epsilon) u + \epsilon \rho$ and $\rho$ is the
solution of the hydrodynamic equation. If $I_{[0,T]}(\cdot|\gamma)$
were convex, to proof that $I_{[0,T]}(\pi^\epsilon|\gamma)$ converges
to $I_{[0,T]}(\pi|\gamma)$ would be straightforward. One direction
would be obtained by the lower semi-continuity of the functional and
the other from the convexity.

The proof of this step presented in \cite{qrv, blm} (cf. \cite[Lemma
5.5]{blm}) for non-convex rate functionals is based on the
representation of a linear functionals given in \cite[Lemma 4.8]{blm},
and relies on Lemmata 5.2 and 5.3 in \cite{blm}. We present here a
much simpler and direct proof which might be useful in other contexts.

The second difficulty in the proof of the
$I_{[0,T]}(\cdot|\gamma)$-density arises, as in \cite{fgln}, from the
presence of a boundary term in the rate
functional. Propositions~\ref{p02}--\ref{p04} state that the cost
$I_{[0,T]}(\pi)$ of a trajectory $\pi\in D([0,T], \mc M)$ can be
expressed as the sum $I^{(1)}_{[0,T]}(\pi) + I^{(2)}_{[0,T]}(\pi)$,
where the first term accounts for the cost due to evolution in the
interior of the interval $[-1,1]$, while the second one for the
evolution at the boundary.

This decomposition is needed in Lemma \ref{l06} for trajectories
bounded away from $0$ and $1$ and presented in Propositions~\ref{p02},
\ref{p03} under this assumption. In \cite{fgln} (cf. equation (4.2))
the trajectory is assumed to be continuous and smooth in time. The
lack of regularity of the trajectory is responsible for the presence
of cross terms and a convoluted expression for the rate
functionals. Assuming regularity, Propositions~\ref{p04} presents a
similar formula to the one obtained \cite{fgln}.

The decomposition of the rate functional requires the following
property of trajectories.  Fix a path $\pi (t,dx) = u(t,x)\, dx$ in
$D([0,T], \mc M)$ with finite rate function, $I_{[0,T]} (u)<\infty$.
It is known, at least since \cite{blm}, that such a path is weakly
continuous in time. More precisely, that for each $g\in C^1([-1,1])$,
the map $t\mapsto \<\pi_t,g\>$ is uniformly continuous (cf. Lemma
\ref{l04}). Proposition \ref{th: IPP} asserts that this map is,
actually, absolutely continuous and its derivative belongs to
$L^1$. This property permits an integration by parts on the right-hand
side of \eqref{07} when $H(t,x) = h(t)\, g(x)$ for some
$h\in C^1([0,T])$, $g\in C^1([-1,1])$ and plays a crucial role in the
decomposition of the rate functional.

\begin{rem}
As in \cite{flm, fgln} and in contrast to \cite{kov, blm}, the large
deviations principle is formulated here for the empirical measure and
not for the empirical density. More precisely, the empirical measure
$\pi(t)$, defined in \eqref{13}, is a sum of Dirac measures, while in
\cite{kov, blm} $\pi(t)$ is defined as a function taking values in the
set $\{0 , 1\}$. Defining $\pi(t)$ as a singular measure requires to
prove that $I_{[0,T]} (\pi) = +\infty$ if $\pi_t$ is not absolutely
continuous with respect to the Lebesgue measure. The proof presented
in Subsection 6.3 of \cite{flm} applies to the present context.
\end{rem}

The paper is organized as follows. In Section \ref{sec3}, we prove the
hydrodynamic limit. In Section \ref{sec4}, we present the main
properties of the rate functional $I_{[0,T]}$ and, in Section
\ref{sec5}, its decomposition. In Section \ref{sec6}, we prove the
$I_{[0,T]}$-density of the set $\Pi_\gamma$, and, in Section
\ref{sec7} the large deviations principle. In Section \ref{sec8}, we
discuss uniqueness of weak solutions of the equation \eqref{10}.

\section{Hydrodynamic limit}
\label{sec3}

In Section \ref{sec8}, we prove uniqueness of weak solutions of the
hydrodynamic equation. The proof of Theorem \ref{t01} is similar to
the classical one presented in Section 5 in \cite{kl}. The slight
difference lies in the energy estimate.

\begin{lem}
\label{l03}
Fix a profile $\rho_0\colon [-1,1] \to [0,1]$, and a sequence of
probability measure $(\mu^N)_{N \in \mathbb{N}}$ associated to
$\rho_0$. Consider a dense family of functions $\{H_l\}_{l \geq 1}$ in
$C^{0,1}([0,1]\times [-1,1])$. There exist a constant $K_0$ such that
\begin{align*}
\underset{N \to + \infty }{\mbox{ lim sup }} \mathbb{E}_{\mu^N}
\Big [ \, \underset{1 \leq j \leq k}{\max} \Big\{
\int_0^T W_N(H_j(s), \eta_s) \, ds\Big \}
\,\Big ] \, \leq\,  K_0
\end{align*}
for every $k \geq 1$, where
\begin{align*}
W_N(H, \eta)  
=\sum_{x}
H\left ( \frac{x}{N} \right ) (\eta(x)- \eta(x+1))- \frac{1}{2 A N}\sum_{x}
H\left ( \frac{x}{N} \right )^2 (\eta(x)- \eta(x+1))^2 \;,
\end{align*}
the sum is performed over all $-N+1 \le x \le N-2$,
and $A = \min \{ 1,1+2a\}$.
\end{lem}

Since $r_{x,x+1} \leq A$ and the Bernoulli product measure is
reversible for the process with generator $L_{N,0}$ (\cite{{flm}}),
the proof is similar to that of Lemma 5.7.3 in \cite{kl}. As stated in
\cite[Theorem 5.7.1]{kl}, Lemma \ref{l03} ensures that all limit
points of the sequence $\bb Q^N_{\mu^N}$ are concentrated on paths
with finite energy.

\section{The rate function}
\label{sec4}

In this section, we present some properties of the rate function
$I_{[0,T]}(.)$ and the energy $\mc Q(\cdot)$. These properties are
crucial for proving Theorem \ref{t02}.

\subsection*{The energy $\mathcal{Q}$}

The proof of the next result can be found in
Section 4 of \cite{blm}.

\begin{proposition}
\label{p06}
The functional $\mathcal{Q}$ is convex and lower semicontinuous. Let
$\pi \in D([0,T], \mathcal{M})$ such that $\pi(t,dx)=u(t,x)dx$. Then $u$
has finite energy if, and only if,
$u \in L^2([0,T],\mathcal{H}^1(\Omega))$ and
$\displaystyle \int_0^T\int_{-1}^1 \frac{\vert\nabla u(t,x)\vert^2}{
\chi(u(t,x))}dxdt<+ \infty$. In this case
$\displaystyle \mathcal{Q}( \pi)=\int_0^T\int_{-1}^1 \frac{\vert\nabla
u(t,x)\vert^2}{\chi(u(t,x))}dxdt$.
\end{proposition}

\begin{corollary}
\label{cor: hydro ener fini}
The weak solution of the hydrodynamic equation has finite energy.
\end{corollary}

\begin{proof}
In view of Lemma B.2, the continuity of $D(\cdot)$ and the boundedness
of the weak solution of the hydrodynamic equation, the proof of
Corollary \ref{cor: hydro ener fini} is similar to the one of Lemma
B.5 in \cite{fgln}.
\end{proof}

\subsection*{The rate functional $I_{[0,T]}$}

In this section, we present certain properties of the rate function
$I_{[0,T]}$: we establish that it is lower semicontinuity and that its
level sets are compact. The proof of the next result is similar to the
one of Lemma 4.1 in \cite{flm}, Lemma 3.1 in \cite{fgln}.

\begin{lem}
\label{l04}
Fix $T>0$. For each $M>0$, $g$ in $C^{1}([0,1])$ and $\epsilon>0$,
there exists $\delta>0$ such that
\begin{equation*}
\sup_{|t-s|\le \delta} \; \sup_{\pi : I_{[0,T]}  (\pi) \le M}\,
\big| \, \< \pi_t, g\> - \< \pi_s , g\> \, \big| \;\le\; \epsilon\;.
\end{equation*}
In particular, $\pi$ belongs to $C([0,T],\mc M^0)$.
\end{lem}


As $D$ is continuous and $u$ is bounded, the proofs of the next two
results are similar to the ones of Proposition 3.5 and Corollary 3.6
of \cite{fgln}.  

\begin{lem}
\label{l05}
There exists a constant $C_0>0$ such that:
\begin{align*}
\int_0^T \int_{-1}^1 \frac{|\nabla u(t,x)|^2}{\chi(u(t,x))}
\leq C_0(I_{[0,T]}(u)+1)    
\end{align*}
for all $u \in D([0,T], \mathcal{M}^0)$.
\end{lem}

\begin{corollary}
\label{cor: I hydro nul}
Let $\pi \in D([0,T], \mathcal{M}^0)$ such that
$\pi(t,dx)=u(t,x)dx$. Then, $u$ is the weak solution of the
hydrodynamic equation (\ref{10}) if, and only if,
$I_{[0,T]}(\pi|\rho_0)=0$.
\end{corollary}

Let $(E_q)_{q \geq 0}$ be the level set of the rate function $I_{[0,T]}(.)$:
\begin{align*}
E_q = \{\pi \in D([0,T], \mathcal{M}) | I_{[0,T]}(\pi) \leq q\}
\end{align*}

\begin{thm}
\label{thm: semi cont de I}
Fix $T>0$. Then,
$I_{[0,T]}\colon D([0,T], \mathcal{M}^0) \to [0, \infty]$ is lower
semicontinuous and has compact level sets.
\end{thm}

\begin{proof}
The proof of the lower semicontinuity is similar to the one of Theorem
2.4 of \cite{fgln}.  We turn to the compactness of $E_q$ in
$D([0,T],\mathcal{M})$. Fix $q \geq 0$.  By lower semicontinuity,
$E_q$ is a closed set. It remains to prove the relative compactness.
By Lemma \ref{l04}, $E_q \subset C([0,T],\mathcal{M})$. This follows
from Ascoli theorem.

Since for each $\pi(t,dx)=u(t,x)dx \in E_q$, $0 \leq u \leq 1$, for
every sequence $(\pi^n)_{n \in \mathbb{N}}$ in $E_q$, $t \in [0,T] $,
there exists a subsequence $(\varphi(n))_{n \in \mathbb{N}}$ such that
$(\pi_t^{\varphi(n)})_{n \in \mathbb{N}}$ converges. On the other
hand, by definition of the distance $d(.,.)$ in $\mathcal{M}$, by
(1.1) in Section 4.1 of \cite{kl}, and by Lemma \ref{l04},
\begin{align*}
\underset{\delta \to 0}{\lim}
\sup_{\pi \in E_q }
\underset{|r-s|\leq \delta}{\mbox{ sup }} d(\pi_s,\pi_r)=0\;.
\end{align*}
This prove that  $E_q$ is relatively compact.
\end{proof}

\begin{remark}
Fix a profile $\gamma\colon [-1,1] \to [0,1]$. The same proof yields
that the rate function $I_{[0,T]}(.|\gamma)$ is lower semicontinuous
and has compact level sets. Details are given in Theorem 3.2.24 in \cite{these}.
\end{remark}

Next two results show that the cost of a trajectory in an interval
$[0,S] \cup [S,T]$ is equal to the sum of its cost on each interval.
Denote by $\color{blue} C_K^{1,2}(]0,T[\times [-1,1])$ the space of
functions in $C^{1,2}([0,T]\times [-1,1])$ with compact support in
$]0,T[\times [-1,1]$.

\begin{lem}
\label{lem: I bord nul}
Fix $T>0$. Let $\pi(t,dx)=u(t,x)dx$ such that $I_{[0,T]}(\pi)<+\infty$. Then,
\begin{align*}
I_{[0,T]}(\pi)=
\underset{H\in C_K^{1,2}(]0,T[\times [-1,1])}{\sup }\hat{J}_H(\pi)
\end{align*}
\end{lem}

\begin{proof} 
Clearly, the right-hand side is bounded by the left-hand side.  We
turn to the reversed inequality.  The proof is presented in
details in Lemma 3.6.1 in \cite{these}.

Fix $\epsilon>0$. Let $\psi_{\epsilon}$ be a function in $C^1([0,T])$
such that $\psi_{\epsilon}(t)=1$ for every
$t \in [\epsilon, T-\epsilon]$ and $\psi_{\epsilon}(t)=0$ for all
$t \in [0,\frac{\epsilon}{2}] \cup  [T-\frac{\epsilon}{2},T]$.
Fix $G \in C^{1,2}([0,T]\times [-1,1])$.  Consider
$G^{\epsilon}=G\psi_{\epsilon}$. Then,
$G^{\epsilon} \in C_K^{1,2}(]0,T[\times [-1,1])$ and by the dominated
convergence theorem and Lemma \ref{l04}:
\begin{align*}
\hat{J}_G(\pi)=\underset{\epsilon \to 0}{\lim}
\hat{J}_{G^{\epsilon}}(\pi) \leq \underset{C_K^{1,2}
(]0,T[\times [-1,1])}{\mbox{ sup }}\hat{J}_H(\pi)\;,
\end{align*}
which completes the proof of the lemma.
\end{proof}

\begin{corollary}
\label{cor: separtion I}
Fix $T>0$. Let $\pi(t,dx)=u(t,x)dx$ such that $I_{[0,T]}(\pi)<+\infty$. 
Then, for each $0 < s < T$,
\begin{align*}
I_{[0,T]}(\pi)=I_{[0,s]}(\pi)+I_{[s,T]}(\pi)   \;.
\end{align*}
\end{corollary}

In the remaining part of this section, we examine the regularity of
the map $t\mapsto \<\pi_t,g\>$ for a trajectory $\pi$ with finite rate
functional and a smooth function $g$.

\begin{lem}
\label{lem : var borne}
Fix $\pi(t,dx)=u(t,x)dx$ such that $I_{[0,T]}(\pi)<+\infty$.  Then,
for each $g \in C^2([-1,1])$, $t \mapsto \<\pi_t,g\>$ is a bounded
variation function.
\end{lem}

\begin{proof}
Fix $\pi(t,dx)=u(t,x)dx$ such that $I_{[0,T]}(\pi)<+\infty$,
$g \in C^2([-1,1])$, $n \in \mathbb{N}^*$, and a partition
$0=t_0 < \cdots < t_n = T$ of the interval $[0,T]$. Consider
$0 \leq r \leq s \leq T$. By definition of the rate functional
$I_{[r,s]}(\cdot)$, and since $\sigma$, $D$, $b$ and $u$ are bounded,
\begin{align*}
\<\pi_s,g\>-\<\pi_r,g\> &\leq
C_0 \int_r^s\int_{-1}^1 |\nabla u(t,x)|\, |\nabla g(x)|\, dx\, dt \\
& +C_0 (s-r)  \int_{-1}^1  \Big\{ (\nabla g(x))^2  + e^{|g|} +1 \Big\}
\, dx \,+\, I_{[r,s]}(\pi)
\end{align*}
for some finite constant $C_0$ which does not depend on $g$ or $\pi$.
Replacing $g$ by $-g$ shows that the absolute value of
$\<\pi_s,g\>-\<\pi_r,g\> $ is bounded by the right-hand side.  Hence,
by Corollary \ref{cor: separtion I}, the linearity of the integral and
Young's inequality $2xy \le x^2 + y^2$,
\begin{align*}
&\sum_{k=0}^{n-1}|\<\pi_{t_{k+1}},g\>-\<\pi_{t_k},g\>|
\leq C_0 \, T\,  \int_{-1}^1  \Big\{ (\nabla g(x))^2  + e^{|g|} +1 \Big\}
\, dx \,+\, \mc Q(\pi) \,+\, I_{[0,T]}(\pi) 
\end{align*}
for some finite constant $C_0$ which does not depend on $g$ or $\pi$.
This completes  the proof of the lemma.
\end{proof}

\begin{prop}
\label{th: IPP}
Fix $\pi(t,dx)=u(t,x)dx$ such that $I_{[0,T]}(\pi)<+\infty$. Then, for
each $g \in C^1([-1,1])$, there exists $P \in L^1([0,T])$ such that,
\begin{align*}
H_T\<\pi_T,g\>-H_0\<\pi_0,g\>-\int_0^TH'_t\, \<\pi_t,g\>\, dt
=\int_0^TH_t \, P_t\, dt
\end{align*}
for all $H \in C^1([0,T])$.
\end{prop}

\begin{proof}
Fix $g \in C^2([-1,1])$. By Lemma \ref{lem : var borne},
$t \mapsto \<\pi_t,g\>$ is a bounded variation function.
Denote by $\nu$ the signed  measure on $[0,T]$ induced by this
function. Integrating by parts yields that
\begin{align*}
H_T\<\pi_T,g\>-H_0\<\pi_0,g\>-\int_0^TH'_t\, \<\pi_t,g\>\, dt
=\int_0^TH_t \, \nu(dt) \;.
\end{align*}
Our goal is to show that $\nu$ is absolutely continuous with
respect to the Lebesgue measure.

For every $t \in [0,T]$, denoted by $V_t$ the total variation of
$\<\pi_.,g\>$ on $[0,t]$.  Since $t \mapsto \<\pi_t,g\>$ is a bounded
variation function, $V$ is a finite increasing function, and
$\mu\colon \displaystyle A \mapsto \int_A dV_t$ is a positive measure.

\smallskip
\noindent{\it Claim 1:}  The measure $\mu$ has no atoms.
\smallskip

Fix $s \in [0,T]$.  Since $V$ is a bounded variation function, the set
$D = \{s: \mu (\{s\}) >0\}$ is countable. Define sequences
$(a_k)_{k \in \mathbb{N}}$ and $(b_k)_{k \in \mathbb{N}}$ such that
$0 \leq a_k\leq a_{k+1}\leq b_{k+1} \leq b_k\leq T$,
$\underset{k \to \infty}{\lim} a_k=\underset{k \to \infty}{\lim}
b_k=s$ and $\mu(\{b_k\})=\mu(\{a_k\})=0$ for all $k \in \mathbb{N}$.
Then, $0 \leq \mu(\{s\}) \leq \mu([a_k,b_k])=V_{b_k}-V_{a_k}$.

Fix $\lambda>0$. By definition of $V$, for each $\epsilon>0$, there
exists a subdivision $(t_i)_{1 \leq i \leq n}$ of $[a_k,b_k]$ such
that:
\begin{align*}
\lambda \mu(\{s\}) \;=\; 
\lambda(V_{b_k}-V_{a_k})
\leq \lambda \epsilon + \sum_{i=0}^{n-1}|\<\pi_{t_{i+1}},\lambda g\>
-\<\pi_{t_i},\lambda  g\>| \;.
\end{align*}
By the proof of the previous lemma, the right-hand side is bounded by
\begin{equation*}
\begin{aligned}
& \lambda \epsilon +   C_0 \,\lambda\, \int_{a_k}^{b_k} \int_{-1}^1 |\nabla u(t,x)|\,
|\nabla g(x)|\, dx\, dt  \\
&\quad \,+\, C_0 \, (b_k-a_k)\,  \int_{-1}^1
\Big\{ (\lambda \nabla g(x))^2  + e^{\lambda |g|} +1 \Big\}
\, dx \,+\, I_{[a_k,b_k]}(\pi)
\end{aligned}
\end{equation*}
for some finite constant $C_0$ which does not depend on $\lambda$
$\epsilon$, $g$ or $\pi$. Let $\epsilon\to 0$, and then
$k\to + \infty$ to conclude that
\begin{align*}
\mu(\{s \}) \,\leq\, \frac{1}{\lambda} \, I_{[0,T]}(\pi) 
\end{align*}
for all $\lambda>0$. It remains to let $\lambda\to\infty$ to complete
the proof of Claim 1. 

\smallskip
\noindent{\it Claim 2:}  The measure $\mu$ is absolutely continuous. 
\smallskip

Fix a set $K \subset [0,T]$ of zero Lebesgue measure, $\epsilon>0$ and
$\lambda>0$.  Let $G$ be an open set such that $K \subset G$ and
$\Lambda (G) \leq \epsilon_1$ where $\epsilon_1$ will be fixed
afterwards and $\Lambda$ represents the Lebesgue measure. Let
$G= \cup_k ]a_k,b_k[$. By the same computation as
before, 
\begin{align*}
0 \leq \lambda \, \mu(K) \,\leq\,  \lambda \, \mu(G)
& \le    C_0 \,\lambda\, \int_{G} \int_{-1}^1 |\nabla u(t,x)|\,
|\nabla g(x)|\, dx\, dt \\
&\quad \,+\, C_0 \, \Lambda (G) \,  \int_{-1}^1
\Big\{ (\lambda \nabla g(x))^2  + e^{\lambda |g|} +1 \Big\}
\, dx\, dt \,+\,  I_{[0,T]}(\pi)
\end{align*}
for some finite constant $C_0$ which does not depend on $\lambda$
$\epsilon$, $g$ or $\pi$.  Choose $\lambda>0$ and then
$\epsilon_1=\epsilon_1(\lambda,\epsilon)>0$ so that
$(1/\lambda) \, I_{[0,T]}(\pi) \leq \epsilon/2$ and
\begin{equation*}
 C_0 \, \int_{G} \int_{-1}^1 |\nabla u(t,x)|\,
 |\nabla g(x)|\, dx\, dt
 \,+\, \frac{C_0}{\lambda} \, \Lambda (G) \,  \int_{-1}^1
\Big\{ (\lambda \nabla g(x))^2  + e^{\lambda |g|} +1 \Big\}
\, dx\, dt \, \le \, \frac{\epsilon}{2}\;\cdot
\end{equation*}
Therefore, $\mu(K)\leq \epsilon$ for every $\epsilon >0$ so that
$\mu(K)=0$, and Claim 2 is proved. 

It follows from Claim 2 that $\nu$ is absolutely continuous, and there
exists $P\in L^1([0,T])$ such that $\nu(dt) = P_tdt$. This completes
the proof of the proposition.
\end{proof}

\begin{remark}
One can show that $P  \in L^1\log L^1$. 
\end{remark}

\section{Deconstructing the rate functional}
\label{sec5}

The main results of this section, stated in Propositions \ref{p02} and
\ref{p03} below, show that the rate function $I_{[0,T]}(\,\cdot\,)$
can be decomposed as the sum of two rate functions. The first one
measures the cost of the trajectory due to its evolution in the bulk,
while the second one measures the costs due to the boundary
evolution. This decomposition of the rate function is an important
tool in the proof that any trajectory $u$ with finite rate function
can be approximated by a sequence of regular trajectories
$(u^n:n\ge 1)$ in such a way that
$I_{[0,T]}(u^n\,|\gamma) \to I_{[0,T]}(u\,|\gamma)$, the content of
the next section.

\subsection*{Weighted Sobolev spaces}

Recall that $\Omega = [-1,1]$, and let
$\color{blue} \Omega_T = [-1,1]\times [0,T]$.  For a \emph{bounded}
positive function $f: \Omega \to \bb R_+$ (resp.\
$f : \Omega_T \to \bb R_+$), denote by $L^2(f)$ (resp.\ $\bb L^2(f)$)
the Hilbert space of (equivalence classes of) measurable functions
$\{H: \Omega \to \bb R : \int_\Omega H(x)^2 f(x) dx < \infty\}$
(resp.\
$\{H: \Omega_T \to \bb R : \int_{\Omega_T} H(t,x)^2 f(t,x) dt dx <
\infty\}$) endowed with the scalar product $\< \cdot, \cdot\>_f$
(resp.\ $\<\!\< \cdot , \cdot \>\!\> _f$) induced by
\begin{equation*}
\< H \,,\,  G  \> _f \;=\; \int_\Omega H(x) \, G(x)\, f(x)\, dx \;, 
\quad \<\!\< H \,,\,  G  \>\!\> _f \;=\; \int_0^T
\< H_t , G_t\>_{ f(t,\cdot)}\; dt \;.
\end{equation*}
The norm associated to the above scalar products is denoted by $\Vert
\cdot\Vert_f$.  When $f=1$ we omit the index $f$ and denote the spaces
$L^2(1)$, $\bb L^2(1)$ by $L^2(\Omega)$, $L^2(\Omega_T)$,
respectively.

Since $f$ is bounded, $C^\infty_K(\Omega_T)$ is dense in $\bb L^2(f)$.
Moreover, the space of bounded linear functionals on $\bb L^2(f)$ can
be identified with $\bb L^2(1/f)$: any bounded linear functional
$\ell$ on $\bb L^2(f)$ can be represented as
\begin{equation}
\label{f04b}
\ell (G) \;=\; \int_0^T  \<H_t \, , \, G_t\> \; dt
\end{equation}
for some $H$ in $\bb L^2(1/f)$. Indeed, by Riesz's representation
theorem, for each bounded linear functional $\ell$ on $\bb L^2(f)$,
there exists a unique element $\hat H =\hat H_\ell \in \bb L^2(f)$
such that $\ell(G) = \<\!\< \hat H \,,\, G \>\!\> _f$. Let $H= f \,
\hat H$. Clearly, $H$ belongs to $L^2(1/f)$ and we obtain the
representation claimed above.

Denote by $\color{blue} \mc H^{1}_0 (f)$ the Hilbert
spaces induced by the set $C^\infty_K (\Omega_T)$ endowed with the
scalar products, $\<\!\< G,H \>\!\>_{1,f}$ defined by
\begin{equation*}
\<\!\< G,H \>\!\>_{1,f} \;=\; \<\!\< \nabla G, \nabla H
\>\!\>_{f} \;.
\end{equation*}
Let $\color{blue} \Vert \cdot \Vert_{1, f}$ be the norm
associated to the scalar product
$\<\!\<\cdot, \cdot \>\!\>_{1,f}$.  Denote $\mathcal{H}^{-1}(f)$
the dual of $\mathcal{H}_0^{1}(f)$, it is a Hilbert space equipped
with the norm $\| \cdot \|_{-1,f} $ defined by
\begin{align*}
\|L\|_{-1,f}^2=\underset{G \in C_K^{\infty}(]0,T[\times
]-1,1[)}
{\mbox{ sup }} \{ \, 2 L(G) \,-\, \|G\|^2_{1,f}\, \} \;.
\end{align*}
In this formula, $L(G)$ stands for the value of the linear form
$L$ at $G$.  The proof of the next proposition is similar to the one
of Lemma 4.8 in \cite{blm} and is therefore omitted.

\begin{proposition}
\label{prop: norme -1}
A linear functional
$L\colon \mathcal{H}_0^1(\sigma(u)) \to \mathbb{R}$ belongs to
$\mathcal{H}^{-1}(\sigma(u))$ if and only if there exists
$P \in L^2(\sigma(u)^{-1})$ such that $L(H)=\<\<P, \nabla H\>\>$ for
every $H \in C^{\infty}_K(]0,T[\times ]-1,1[)$. In this case,
\begin{align*}
\|L\|_{-1, \sigma(u)}^2=\int_0^T
\left \{\<P_t,P_t\>_{\sigma(u_t)^{-1}}-c_t \right \}dt \mbox{ , }
\end{align*}
\\where $c_t=\<P_t\sigma(u_t)^{-1}\>^2
\<\sigma(u_t)^{-1}\>^{-1}\, \mathbb{1}_{\{\<\sigma(u_t)^{-1}\><\infty\}}$ .
\end{proposition}

\begin{lem}
\label{l01} 
Fix $T>0$.  Let $\pi(t,x)=u(t,x)dx$ such that
$I_{[0,T]}(\pi)<+\infty$. Then,
$H \mapsto \int_0^T \<\pi_t, \partial_t H_t\>dt$ is continuous on
$\mathcal{H}_0^1(\sigma(u))$ and there exists
$P \in L^2(\sigma(u)^{-1})$ such that
\begin{equation*}
-\int_0^T \<\pi_t, \partial_t H_t\> dt
\,=\,
\<\<\nabla H,P\>\> \;\; \text{for all}
\;\; H \in C^{\infty}_K(]0,T[\times ]-1,1[)\;.
\end{equation*}
\end{lem}

\begin{proof}
Define the function
$A\colon C^{\infty}_K(]0,T[\times [-1,1]) \mapsto \mathbb{R}$ by
\begin{align*}
A(H)= -\int_0^T \<\pi_t, \partial_t H_t\> dt \mbox{ . }
\end{align*}
Let $H \in C^{\infty}_K(]0,T[\times [-1,1])$. By definition of
$\hat{J}_H(\pi)$:
\begin{align*}
A(H)&= \hat{J}_H(\pi)+\int_0^T \int_{-1}^1\sigma(u(t,x))(\nabla
H(t,x))^2dx dt
\\
& - \int_0^T\int_{-1}^1D(u(t,x))\nabla u(t,x) \nabla H(t,x)dx dt
\,+\, \int_0^T B(u_t,H_t)dt \mbox{ . }
\end{align*}
Since $H(t,.)$ has compact support in $]-1,1[$, the last term on the
right-hand side vanishes. Replace $H$ by $RH$ for some $R>0$.  We may
estimate $\hat{J}_{RH}(\pi)$ by $I_{[0,T]}(\pi)$. By linearity of $A$,
since $D$ is bounded,
and $2xy\leq ax^2+\frac{1}{a}y^2$ for all $a>0$, 
\begin{align*}
A(H)\,\le\, \frac{1}{R} \Big\{ I_{[0,T]}(\pi)
+2 R^2 \|H \|^2_{1,\sigma(u)}
\,+\, C_0 \int_0^T \int_{-1}^1
\frac{(\nabla u(t,x))^2}{\sigma(u(t,x))} dx dt \, \Big \}  
\end{align*}
for some finite constant $C_0$ which may change from line to line.  

By Lemma \ref{l05}, the previous expression is less
than or equal to
\begin{align*}
\frac{C_0}{R}\left  \{I_{[0,T]}(\pi)+1\right \}
+2R\| H\|_{1,\sigma(u)}^2
\;. 
\end{align*}
Replacing $H$ by $-H$, and optimizing over $R$ yields that
\begin{align*}
A(H)^2 \, \leq\, C_0\, (I_{[0,T]}(\pi)+1)\,
\| H\|_{1,\sigma(u)}^2 \;.
\end{align*}

This proves that the linear functional $A$ is bounded in
$\mathcal{H}_0^1(\sigma(u))$. To complete the proof, it remains to
recall the assertion of Proposition \ref{prop: norme -1}.
\end{proof}

\subsection*{First decomposition}

Until the end of this section, fix $T>0$ and $\pi(t,dx)=u(t,x)dx$ such
that $I_{[0,T]}(\pi)<+\infty$.  Recall the definition of the linear
operator $L_0$ introduced in \eqref{07}. Denote
$L\colon C^{1,2}([0,T]\times [-1,1]) \to \mathbb{R}$ the functional
given by
\begin{align}
\label{eq: L}
L(H)&=L_0(H) + \int_0^T\int_{-1}^1D(u(t,x))\nabla u(t,x) \nabla H(t,x)dx dt \notag\\
&+\int_0^T \Bigl \{(u_t(1)-{\mf b})H_t(1)+(u_t(-1)-{\mf a})H_t(-1) \Bigl \} dt \mbox{ , }
\end{align} 
Denote by $\color{blue} C_0^{1,2}([0,T]\times [-1,1])$ the space of
functions $H \in C^{1,2}([0,T]\times [-1,1])$ such that
$H_T=H_0=0$. We use the same notation $L_0$, $L$ when these
functionals are defined on $C_0^{1,2}([0,T]\times [-1,1])$.

For functions $g\in C^2( [-1,1])$,  $H \in C_0^1([0,T])$, denote by
$Hg$ the function defined by $Hg (t,x) = H(t)g(x)$. Next result is an
elementary consequence of Proposition  \ref{th: IPP}. 

\begin{lem}
\label{l02} 
Fix $g\in C^2( [-1,1])$.  Then, there exists $P^g \in L^1([0,T])$ such
that 
\begin{align*}
L_0(Hg) \,=\, \int_0^TH(t)P^g_t dt
\end{align*}
for every $H \in C_0^1([0,T])$.
\end{lem}

Next result requires some notation.  Denote by
$P \in L^2(\sigma(u)^{-1})$ the function introduced in Lemma \ref{l01}
and such that:
\begin{align}
\label{eq: def P}
-\int_0^T \<\pi_t, \partial_t H_t\>dt=\<\<\nabla H,P\>\>    
\end{align}
for every $H \in C_K^{\infty}(]0,T[\times ]-1,1[)$. If we replace
$P_t$, by $P_t-\< P_t/\sigma(u_t)\> / \<1/\sigma(u_t)\>$, we may
assume that $\< P_t/\sigma(u_t)\>=0$ for all $0\le t\le T$.

Let $X$, $\mathbb{1} \colon [-1,1] \to \mathbb{R}$ be the functions
given by $X(x) = x$, $\mathbb{1} (x)=1$. By Proposition \ref{th: IPP},
there exist $P^g$ in $L^1([0,T])$ such that 
\begin{align*}
-\int_0^T H'(t) \<\pi_t,g\>dt \,=\, \int_0^T H(t)\, P_t^g\, dt 
\end{align*}
for $g=X$, $\mathbb{1}$ and every $H \in C_0^1([0,T])$.

Denote by $M\in L^2(\sigma(u)^{-1})$, $g,h\in L^1([0,T])$
the functions defined by:
\begin{gather}
\label{04}
M_t=P_t+D(u_t)\nabla u_t \mbox{ , } \\
g(t) \,=\, 2 \,
\Big\{\, P^{\mathbb{1}}(t)-P^X(t)-\<D(u_t)\nabla u_t\>+2(u_t(-1)-{\mf a})
+ \<M_t\> - 2\<M_t\sigma(u_t)^{-1}\>\mathcal{S}_t \,\Big\} \nonumber
\\
h(t) \,=\, 2 \,
\Big\{\, P^{\mathbb{1}}(t)+ P^X(t)+ \<D(u_t)\nabla u_t\>+2(u_t(1)-{\mf b})
-\<M_t\> + 2 \<M_t \sigma(u_t)^{-1}\>\mathcal{S}_t \,\Big\} \nonumber \mbox{ , }
\end{gather}
where $\color{blue} \mathcal{S}_t= 1/ \<\sigma(u_t)^{-1}\>$. Lastly,
denote by $\Psi\colon [0,1]^2 \times \bb R^2 \to \bb R_+$ the function
given by:
\begin{align*}
\Psi(r,s,\alpha ,\beta)
&=4{\mf a}(1-r)(e^{\alpha}-\alpha-1)
+4r(1-{\mf a})(e^{-\alpha}+\alpha-1) \\
&+4{\mf b}(1-s)(e^{\beta}-\beta-1)
+4s(1-{\mf b})(e^{-\beta}+\beta-1) \mbox{ . }
\end{align*}
Note that $\Psi$ is non-negative, linear in $(r,s)$ and convex in
$(\alpha,\beta)$.

\begin{prop}
\label{p02}
$I_{[0,T]}(\pi)=I_{[0,T]}^{(1)}(\pi)+I_{[0,T]}^{(2)}(\pi)$ where:
\begin{align*}
I_{[0,T]}^{(1)}(\pi) \,=\, \frac{1}{4}
\, \int_0^T \Big\{\,  \<M_t,M_t\>_{\sigma(u_t)^{-1}}
\,-\, \<M_t\sigma(u_t)^{-1}\>^2\mathcal{S}_t \,\Big\}\,
dt \;,
\end{align*}
\begin{align*}
I_{[0,T]}^{(2)}(\pi) \;=\; 
\frac{1}{4} \, \underset{\alpha, \beta}{ \sup }
\int_0^T \Big \{ \, \alpha(t)\, g(t) + \beta(t)\, h(t) \,
\, -\, \Upsilon_t (u_t(-1),u_t(1),\alpha(t),\beta(t)) \Big \} \, dt   \;.
\end{align*}
where the supremum is carried out over all $\alpha$,
$\beta\in C_0^1([0,T])$, and
\begin{equation}
\label{03}
{\color{blue} \Upsilon_t(r,s,\alpha,\beta)} \,:=\,
4\, (\alpha -\beta )^2 \, \mathcal{S}_t
\,+\, \Psi(r,s,\alpha ,\beta) \;.
\end{equation}
\end{prop}

\begin{proof}
By Lemma \ref{lem: I bord nul}, and replacing in the supremum the
function $H$ by $H/2$, yields that
\begin{align*}
I_{[0,T]}(\pi) \,=\,
\frac{1}{4}\underset{H}{\mbox { sup }}
\Big \{ 2L(H)-\int_0^T \|\nabla H_t\|_{\sigma(u_t)}^2 dt
-\int_0^T \Psi(u_t, H_t)\, dt \Big \}\;,
\end{align*}
where $L$ has been introduced in (\ref{eq: L}) and the supremum is
carried over all functions $H \in C_K^{1,2}(]0,T[\times [-1,1])$.  In
this formula and below, $\Psi(u_t,H_t) = \Psi (u_t(-1),u_t(1),$
$H_t(-1)/2, H_t(1)/2)$.
 
For each $H \in C^{1,2}_0([0,T]\times [-1,1])$, decompose $H$ as
$H=H^{(0)}+H^{(1)}$ where
\begin{align*}
H^{(1)}(t,x)=\frac{H(t,1)-H(t,-1)}{2}x+\frac{H(t,-1)+H(t,1)}{2} \mbox{ . }
\end{align*}
Note that $H^{(0)}(t,\pm1)=0$, $H^{(1)}(t,\pm 1)=H(t,\pm 1)$ for all
$0 \leq t \leq T$. By linearity of $L$, the previous expression is
equal to
\begin{align*}
& \frac{1}{4} \underset{H^{(1)}}{\mbox { sup }}
\Big \{2L(H^{(1)})-\int_0^T \|\nabla H_t^{(1)}\|_{\sigma(u_t)}^2 dt
-\int_0^T \Psi (u_t, H_t^{(1)})\, dt  \\
&  \quad + \underset{H^{(0)} }{\mbox { sup }}
\Big[\, 2 L(H^{(0)})-2\int_0^T
\<\nabla H_t^{(0)},\nabla H_t^{(1)}\>_{\sigma(u_t)}dt
-\int_0^T \|\nabla H_t^{(0)}\|_{\sigma(u_t)}^2 dt\, 
\Big ]\,  \Big \} \;.
\end{align*} 

We first compute the supremum over the function $H^{(0)}$.  Since
$H^{(0)}$ belongs to $C^{1,2}_0([0,T] \times [-1,1])$ and vanishes at
$x=\pm 1$, and since the set of smooth functions
$C_K^{\infty}(]0,T[\times ]-1,1[)$ is dense in
$C_0^{1,2}([0,T]\times ]-1,1[)$, the supremum over $H^{(0)}$ is equal
to:
\begin{align*}
\underset{H \in C_K^{\infty}(]0,T[\times ]-1,1[ )}{\mbox { sup }}
\Big\{ \, 2L(H)-2\int_0^T\<\nabla H_t,\nabla
H_t^{(1)}\>_{\sigma(u_t)}dt-\int_0^T \|\nabla H_t\|_{\sigma(u_t)}^2 \,
dt \, \Big \} \mbox{ . }
\end{align*}
By the definition of $P$, given in (\ref{eq: def P}), and the one of
$M$, introduced in (\ref{04}), the previous supremum is equal to
\begin{align*}
\underset{H \in C_K^{\infty}(]0,T[\times ]-1,1[ )}{\mbox { sup }}
\Big\{ 2\int_0^T \big\< \, M_t-\sigma(u_t)\, \nabla H_t^{(1)}
\,,\, \nabla H_t \,\big\>
\, dt-\int_0^T \|\nabla H_t\|_{\sigma(u_t)}^2 \, dt \Big \} \;.
\end{align*}
Since $M_t- \sigma(u_t)\nabla H_t^{(1)}$ belongs to
$L^2(\sigma(u)^{-1})$, by Proposition \ref{prop: norme -1}, the
previous supremum is equal to :
\begin{align*}
\int_0^T \Big\{\, \Vert\, M_t-\sigma(u_t)\nabla H_t^{(1)}
\Vert^2_{\sigma(u_t)^{-1}}
\,-\, \<\, [M_t-\sigma(u_t)\nabla H_t^{(1)}]\, \sigma(u_t)^{-1}\>^2
\, \mathcal{S}_t \,\Big\}\, dt \;.
\end{align*}
Expanding the squares and by definition of $I_{[0,T]}^{(1)}(\pi)$, this
expression is equal to
\begin{align*}
& 4\, I_{[0,T]}^{(1)}(\pi)+\int_0^T \|\nabla
H_t^{(1)}\|_{\sigma(u_t)}^2 \, dt 
\,-\, \int_0^T\<\nabla H_t^{(1)}\>^2 \, \mathcal{S}_t \, dt \\
& \quad -\, 2\, \int_0^T \<M_t,\nabla H_t^{(1)}\>dt
\,+\, 2\, \int_0^T\<M_t\sigma(u_t)^{-1}\>\<\nabla H_t^{(1)}\>
\, \mathcal{S}_t\, dt \;.
\end{align*}
Note that the term
$\int_0^T \|\nabla H_t^{(1)}\|_{\sigma(u_t)}^2 \, dt$ cancels with one
which appears in the formula for $I_{[0,T]}(\pi)$

Up to this point, we proved that $I_{[0,T]}(\pi)$ is equal to
\begin{align*}
I_{[0,T]}^{(1)}(\pi)+\frac{1}{4}
\underset{H^{(1)}}{\mbox { sup }} \Big \{\, & 2L(H^{(1)})-\int_0^T \Psi
\bigl (u_t, H^{(1)}_t)dt 
-\int_0^T\<\nabla H_t^{(1)}\>^2 \mathcal{S}_tdt \\
& -2\int_0^T \<M_t,\nabla H_t^{(1)}\>dt 
+2\int_0^T\<M_t\sigma(u_t)^{-1}\>\<\nabla H_t^{(1)}\>
\mathcal{S}_tdt \Big \} \mbox{ . }
\end{align*}

Let $\alpha(t) = H^{(1)}(t,-1)/2$, $\beta (t) = H^{(1)}(t,1)/2$ so
that $\alpha$, $\beta \in C_0^1([0,T])$, and
$H^{(1)} = (\beta -\alpha) X + (\beta+\alpha) \, \bb 1$, where the
functions $X$, $\bb 1$ have been introduced just before the statement
of the proposition.  Replace $H^{(1)}$ by
$(\beta -\alpha) X + (\beta+\alpha) \, \bb 1$ in the previous
formula. A straightforward computation yields that the expression
inside braces coincides with the one appearing in the definition of
$I_{[0,T]}^{(2)}(\pi)$. This completes the proof of the lemma since
the supremum over $H^{(1)}$ corresponds to one over
$\alpha$, $\beta \in C_0^1([0,T])$.
\end{proof}   

Let $\Phi_t\colon [0,1]^2\times \mathbb{R}^2 \to \mathbb{R}$,
$0\le t\le T$, be given by:
\begin{equation}
\label{05}
{\color{blue} \Phi_t(r,s,x,y)} \,:=\,
\underset{\alpha, \beta\in \mathbb{R}^2}
{\mbox { sup }}
\Bigl \{ \alpha x +\beta y-\Upsilon_t(r,s,\alpha,\beta) \Bigl \}    
\end{equation}
where $\Upsilon_t$ has been introduced in \eqref{03}.  Taking
$\alpha=\beta=0$ yields that $\Phi_t(r,s,x,y)\geq 0$. Moreover, since
$\Psi$ is strictly convex in $(\alpha,\beta)$, by Theorem 26.6 in
\cite{rt}, for each $(x,y) \in \mathbb{R}^2$ there exists
$(\alpha,\beta) \in \mathbb{R}^2$ such that
\begin{align*}
\Phi_t(r,s,x,y)
=\alpha x +\beta y-\Upsilon_t(r,s,\alpha,\beta)  \mbox{ . }
\end{align*}

\begin{prop}
\label{p03}
Fix $\pi(t,dx)=u(t,x)dx$ such that $I_{[0,T]}(\pi)<+ \infty$. Assume
that there exists $\epsilon_0>0$ such that
$\epsilon_0 \le u \le 1-\epsilon_0$. Then, 
\begin{align}\label{eq: I^2 sup int}
I_{[0,T]}^{(2)}(\pi)&=\frac{1}{4}\int_0^T
\Phi_t\bigl(u_t(-1),u_t(1),g(t),h(t)\bigl) \, dt \mbox{ . }
\end{align}
Moreover, there exists a finite constant $C(\epsilon_0)$ such that
\begin{align}
\label{06}
e^{|\alpha(t)|}+e^{|\beta(t)|} \leq C(\epsilon_0)\,
\big \{\, 1+|g(t)|+|h(t)| \, \big \}
\end{align}
for all $0 \leq t \leq T$. In this formula, $g$, $h$ are the functions
introduced in \eqref{04}, and $(\alpha(t),\beta(t))$ the pair which
optimizes the variational problem \eqref{05} with $r=u_t(-1)$,
$s=u_t(1)$, $x=g(t)$, $y=h(t)$.
\end{prop}

\begin{proof} 
We first prove \eqref{06}.  Fix $t \in [0,T]$ such that $|g(t)|$,
$|h(t)|<+\infty$.  Recall the definition of $\alpha(t),
\beta(t)$. Clearly,
\begin{equation*}
U (\alpha) = g(t)-8(\alpha-\beta)\mathcal{S}_t \;\;\text{and}\;\;
V (\beta)  = h(t)+8(\alpha-\beta)\mathcal{S}_t\;,
\end{equation*}
where 
\begin{gather*}
U(\alpha) = 4{\mf a}(1-u_t(-1))(e^{\alpha}-1)
\,-\, 4u_t(-1)(1-{\mf a})(e^{-\alpha}-1) \;,
\\
V(\beta) = 4 {\mf b}(1-u_t(1))(e^{\beta}-1)
-4u_t(1)(1-{\mf b})(e^{-\beta}-1) \;.
\end{gather*} 
Suppose, to fix ideas, that $\alpha \leq 0 \leq \beta$. The other
cases are handled similarly. By definition of $U$,
\begin{align*}
4u_t(-1)\, (1-{\mf a})\, (e^{-\alpha}-1)
\,=\, 4{\mf a}\, (1-u_t(-1))\, (e^{\alpha}-1)-g(t)
+8\, (\alpha-\beta)\, \mathcal{S}_t  \mbox{ . }  
\end{align*}
Since $\epsilon_0 \leq u \leq 1-\epsilon_0$ and $\alpha \leq 0$, there
exists $C(\epsilon) <+ \infty$:
\begin{align*}
e^{-\alpha} \,\leq\, C(\epsilon_0)\,
\big \{ \, 1+|g(t)| + |\alpha| + |\beta| \, \big \} \mbox{ . } 
\end{align*}
A similar argument yields that:
\begin{align*}
e^{\beta} \,\leq\,  C(\epsilon_0)\, 
\big \{\, 1 + |h (t)| + |\alpha| + |\beta| \, \big \} \mbox{ . } 
\end{align*}
Since there exists a finite constants $C_0$ such that
$a \leq C_0 + (1/4) e^a$ for all $a \geq 0$, summing the previous
bounds yields that
\begin{align*}
e^{|\alpha|}+e^{|\beta|}=e^{-\alpha}+e^{\beta}
\leq C(\epsilon_0)\, \big \{\, 1+|g(t)|+|h(t)| \, \big \} 
\end{align*}
as claimed.

We turn to the proof of identity (\ref{eq: I^2 sup int}). It is clear
that
\begin{align*}
I_{[0,T]}^{(2)}(\pi)&\leq \frac{1}{4}\int_0^T 
\Phi_t\bigl(u_t(-1),u_t(1),g(t),h(t)\bigl)\, dt \mbox{ . }    
\end{align*}
To prove the reversed inequality, fix $t \in [0,T]$, and recall the
definition of the pair $(\alpha(t),\beta(t))$ introduced in the
statement of the proposition, and the one of $\Upsilon_t$ given in
\eqref{03}.

\smallskip \noindent {\it Claim A}: Denote by $B([0,T])$ the subset of
bounded functions. Then,
\begin{align*}
& \int_0^T \Phi_t\bigl(u_t(-1),u_t(1),g(t),h(t)\bigl)\, dt \\
&\quad 
\leq \underset{\alpha, \beta\in B([0,T])}{\mbox { sup }} \int_0^T
\Bigl \{ \alpha(t) g(t) +\beta(t) h(t)
-\Upsilon_t\bigl(u_t(-1),u_t(1),\alpha(t),\beta(t)\bigl)\Bigl \} \, dt
\mbox{ . } 
\end{align*}
\smallskip

Let $(\alpha_n(t),\beta_n(t))$ be given by
\begin{align*}
\alpha_n(t)=\alpha(t)\, \mathbb{1}_{\{|\alpha(t)|\leq n\}}
\mathbb{1}_{\{|\beta(t)|\leq n\}} \;,\quad
\beta_n(t)=\beta(t) \, \mathbb{1}_{\{|\alpha(t)|\leq n\}}
\mathbb{1}_{\{|\beta(t)|\leq n\}} \mbox{ . }
\end{align*}
By definition, $(\alpha_n,\beta_n)$ belong to $B([0,T])$.  As
$\alpha(\cdot)$, $\beta(\cdot)$ are finite almost surely,
$\lim_n\alpha_n(t)=\alpha(t)$ and $\lim_n \beta_n(t)=\beta(t)$ almost
surely.  Moreover:
\begin{align*}
\alpha_n(t)g(t) + \beta_n(t)h(t)  -
\Upsilon_t\bigl(u_t(-1),u_t(1),\alpha_n(t),\beta_n(t)\bigl) \geq 0
\end{align*}
because this expression is either equal to
$- \Psi\bigl(u_t(-1),u_t(1),0,0\bigl) \ge 0$ or it is equal to
$\Phi_t\bigl(u_t(-1),u_t(1),g(t),h(t)\bigl) \geq 0$.  Therefore, by
Fatou's Lemma,
\begin{align*}
&\int_0^T 
\Phi_t\bigl(u_t(-1),u_t(1),g(t),h(t)\bigl)dt\\
= &\int_0^T\underset{n \to + \infty}{\mbox{ lim inf }} \Bigl
\{\alpha_n(t) g(t) +\beta_n(t)
h(t)-\Upsilon_t\bigl(u_t(-1),u_t(1),\alpha_n(t),\beta_n(t)\bigl) \Bigl
\} \, dt \\
\leq &\underset{n \to + \infty}{\mbox{ lim inf }}\int_0^T \Bigl
\{\alpha_n(t) g(t) +\beta_n(t)
h(t)-\Upsilon_t\bigl(u_t(-1),u_t(1),\alpha_n(t),\beta_n(t)\bigl) \Bigl
\} \, dt\;.
\end{align*}
Since $(\alpha_n,\beta_n) \in B([0,T])^2$, the previous expression is
bounded by:
\begin{align}\label{eq : f1}
\underset{\alpha, \beta\in B([0,T])}{\mbox { sup }} \int_0^T \Bigl \{
\alpha(t) g(t) +\beta(t)
h(t)-\Upsilon_t\bigl(u_t(-1),u_t(1),\alpha(t),\beta(t)\bigl)\Bigl \}
\, dt \mbox{ , }
\end{align}
as claimed.

Fix $\epsilon>0$ and $(\alpha,\beta)\in B([0,T])^2$. Let
$\varphi\colon \mathbb{R} \to [0,+ \infty[$ be a smooth function such
that ${\rm supp }\, \varphi \subset[-1,1]$,
$\int_{-1}^1\varphi(x)dx=1$ and $\varphi(-x)=\varphi(x)$ for every
$x \in \mathbb{R}$. Let
$\varphi_{\epsilon}(x)=\frac{1}{\epsilon}\varphi(\frac{x}{\epsilon})$,
$\epsilon>0$. Extend the definition of $\alpha$, $\beta$, setting
$\alpha(t)=\beta(t) =0$ for $t<0$, $t>T$. Let
$\alpha_{\epsilon}=\alpha*\psi_{\epsilon}$ and
$\beta_{\epsilon}=\beta*\psi_{\epsilon}$. Since $g,h \in L^1([0,T])$
and $\alpha$, $\beta$, $u$ are bounded:
\begin{align*}
& \underset{\epsilon \to 0}{\lim}
\int_0^T \Big\{\,
\alpha_{\epsilon}(t)g(t)+\beta_{\epsilon}(t)h(t)
-\Upsilon_t\bigl(u_t(-1),u_t(1),\alpha_{\epsilon}(t),
\beta_{\epsilon}(t)\bigl) \,\Big\}\, dt \\
& \quad=\,
\int_0^T\Bigl \{\alpha(t) g(t)+\beta(t) h(t)
-\Upsilon_t\bigl(u_t(-1),u_t(1),\alpha(t),\beta(t)\bigl)\Bigl \}
\, dt   \mbox{ . }
\end{align*}
Since $\alpha_{\epsilon}$ and $\beta_{\epsilon}$ are smooth functions,
(\ref{eq : f1}) is equal to
\begin{align*}
\underset{\alpha,\beta\in C^{\infty}([0,T])}{\mbox{ sup }} 
\int_0^T\Bigl
\{\alpha(t)g(t)+\beta(t)h(t)-\Upsilon_t
\bigl(u_t(-1),u_t(1),\alpha(t),\beta(t)\bigl)
\Bigl \}  \, dt\mbox{ . }
\end{align*}
It remains to approximate smooth functions $\alpha$ and $\beta$ in
$C^{\infty}$ by functions which vanish at the boundary. This is
standard and left to the reader. Details can be found in Theorem 3.7.8 in \cite{these}.
\end{proof}

\begin{remark}
Proposition \ref{p03} can be proved without the condition
$\epsilon_0<u<1-\epsilon_0$. We refer to Section 3.7.2 in \cite{these}.
\end{remark}

The proof of the next result is similar to the one presented in
Section 4 of \cite{fgln}.

\begin{prop}
\label{p04}
Fix $T>0$ and $\pi(t,dx)=u(t,x)dx$ such that $I_{[0,T]}(\pi)<+\infty$.
Assume that there exists $\epsilon_0>0$ such that
$\epsilon_0 \le u(t,x) \le 1-\epsilon_0$ for all $(t,x)$ in
$[0,T] \times [-1,1]$ and that $u(\cdot,x) \in C^{\infty}([0,T])$ for
all $x \in [-1,1]$. Then,
$I_{[0,T]}(\pi)=I_{[0,T]}^{(a)}(\pi)+I_{[0,T]}^{(b)}(\pi)$, where
\begin{gather*}
I_{[0,T]}^{(a)}(\pi)
= \frac{1}{4}\int_0^T \left \{ \|P_t+D(u_t)\nabla
u_t\|_{\sigma(u_t)^{-1}}^2-R_t \right \} \, dt \\
I_{[0,T]}^{(b)}(\pi)
= \frac{1}{4} \,
\int_0^T \Phi\bigl(u_t(-1),u_t(1),a(t),b(t)\bigl) \, dt
\end{gather*}
$P$ is defined in (\ref{eq: def P}), and
\begin{gather*}
\Xi(t,x)= \frac{1}{\<\sigma(u_t)^{-1}\>}\,
\int_{-1}^x \frac{1}{\sigma(u_t(y))}\, dy \;, \quad 
R_t=\left \<\frac{D(u_t)\nabla u_t}{\sigma(u_t)}\right \>^2
\mathcal{S}_t \mbox{ , }  \\
\vphantom{\Big\{}
a(t) \,=\, 4\, \Big\{\, \<\partial_t u_t, (1-\Xi(t))\>
-\<D(u_t)\nabla u_t, \nabla \Xi_t\> + u_t(-1) - \mf a \,\Big\} \;, \\
\vphantom{\Big\{}
b(t) \,=\, 4\, \Big\{\,
\<\partial_t u_t, \Xi(t)\>+\<D(u_t)\nabla u_t, \nabla \Xi_t\>
+ u_t(1) - \mf b \,\Big\}  \;.
\end{gather*}
\end{prop}

\section{$I_{[0,T]}(.|\gamma)$-density}
\label{sec6}

Fix a profile $\gamma\colon [-1,1] \to [0,1]$ satisfying the
hypotheses of Theorem \ref{t02}. The main result of this section
states that any trajectory $\pi \in D([0,T],\mathcal{M})$ with finite
rate function can be approximated by a sequence of smooth trajectories
$(\pi^n)_{n \in \mathbb{N}}$ such that
$ I_{[0,T]}(\pi^n|\gamma) \to I_{[0,T]}(\pi|\gamma)$.

Denote by $\rho$ the weak solution of the hydrodynamic equation
(\ref{10}) with initial condition $\gamma$. Recall from Definition
\ref{d04} that $\Pi_\gamma$ represents the set of paths
$\pi(t,dx) = u(t,x)dx$ in $D([0,T], \mathcal{M})$ such that:
\begin{itemize}
\item[(a)] There exist $\delta>0$ such that $u_t=\rho_t$
on $[0,\delta]$;

\item[(b)]  $u \in C^{\infty}([\delta,T] \times [-1,1])$ and
$u \in C([0,T] \times [-1,1])$

\item[(c)] For each $\delta'>0$, there exists $\epsilon>0$ such that
$\epsilon \le  u(t,x) \le 1-\epsilon$ for all $(t,x) \in [\delta',T]
\times [-1,1]$.
\end{itemize}

The main result of this section reads as follows. 

\begin{prop}
\label{p05}
Let $\gamma$ be a density profile satisfying the hypotheses of Theorem
\ref{t02}. Then, the set $\Pi_\gamma$ is $I_{[0,T]}(.|\gamma)$-dense.
\end{prop}

The proof is divided
in several steps.

\subsection*{ $\mathcal{F}_1$ density}

Let $\mathcal{F}_0$ be the set of all paths $\pi(t,dx)=u(t,x)dx$ in
$D([0,T], \mathcal{M}^0)$ with finite energy and for which there
exists $\delta >0$ for which $u_t=\rho_t$ for $0 \leq t \leq \delta $.
The proof of the next result is similar to one of Lemma 5.3 in
\cite{fgln}.

\begin{lem}
\label{lc01}
The set $\mathcal{F}_0$ is $I_{[0,T]}(.|\gamma)$-dense.
\end{lem}

Let $\mathcal{F}_1$ be the set of all paths $\pi(t,dx)=u(t,x)dx$ in
$\mathcal{F}_0$ with the property that for every $\delta>0$ there
exists $\epsilon>0$ such that $\epsilon \le u(t,x) \leq 1-\epsilon$
for all $(t,x) \in [\delta,T] \times [-1,1]$.

\begin{lem}
\label{lc02}
The set $\mathcal{F}_1$ is $I_{[0,T]}(.|\gamma)$-dense.
\end{lem}

\begin{proof}
Fix $\pi(t,dx)=u(t,x)dx$ in $\mathcal{F}_0$. By Lemma \ref{lc01}, it
is enough to show that there exists a sequence
$(\pi^n)_{n \in \mathbb{N}}$ in $\mathcal{F}_1$ such that
$I_{[0,T]}(\pi^n|\gamma) \to I_{[0,T]}(\pi|\gamma)$ and
$\pi^n\to \pi$.

Let $\pi^{\epsilon}(t,dx)=u^{\epsilon}(t,x)dx$, where
$u^{\epsilon}=(1-\epsilon)u+\epsilon \rho$, $0<\epsilon<1$.  We claim
that $\pi^{\epsilon}$ belongs to $\mathcal{F}_1$. Clearly
$\pi^{\epsilon} \in D([0,T],\mathcal{M}^0)$. By the definition of
$\mathcal{F}_0$, by Corollary \ref{cor: hydro ener fini} and by
convexity of $\mathcal{Q}_{[0,T]}$ (Proposition \ref{p06}),
$\pi^{\epsilon}$ has finite energy because
\begin{align*}
\mathcal{Q}_{[0,T]}(u^{\epsilon})
\,\leq\, \epsilon\, \mathcal{Q}_{[0,T]}(\rho)
\,+\, (1-\epsilon)\, \mathcal{Q}_{[0,T]}(u) < + \infty \;,
\end{align*}
Since $\pi \in \mathcal{F}_0$, there exists $\delta>0$ such that
$u_t=\rho_t$ for every $ 0\leq t \leq \delta$. Hence, by construction
of $\pi^{\epsilon}$, $u^{\epsilon}_t=\rho_t$ for every
$ 0\leq t \leq \delta$. Moreover, by Proposition \ref{p07}, for each
$\delta_1>0$, there exist $\epsilon_1>0$ such that
$\epsilon_1 \leq \rho_t \leq 1- \epsilon_1$, for every
$\delta_1 \leq t \leq T$. Therefore, for $t$ in this interval,
\begin{align*}
\epsilon\epsilon_1 \leq (1- \epsilon)u_t
+\epsilon \rho_t \leq 1- \epsilon
+\epsilon(1- \epsilon_1)=1-\epsilon\epsilon_1 \;.
\end{align*}
This shows that $\pi^{\epsilon} $ belongs to $ \mathcal{F}_1$.

It's clear that $\pi^{\epsilon}$ converges to $\pi$ in
$D([0,T], \mathcal{M})$ as $\epsilon \to 0$.  To conclude the proof,
it remains to show that $I_{[0,T]}(\pi^{\epsilon}|\gamma)$ converges
to $I_{[0,T]}(\pi|\gamma)$. By Theorem \ref{thm: semi cont de I},
$I_{[0,T]}$ is lower semicontinuous, so that
$\liminf_{\epsilon \to 0} I_{[0,T]}(\pi^{\epsilon}|\gamma) \geq
I_{[0,T]}(\pi|\gamma)$. We turn to the reverse inequality,
$\limsup_{\epsilon \to 0} I_{[0,T]}(\pi^{\epsilon}) \leq
I_{[0,T]}(\pi)$.

Recall from \eqref{07} the definition of the functionals $L_0$, $L$.
To stress its dependence on $\pi(t,dx) = u(t,x)\, dx$, we denote them
below by $L_0(u,H)$, $L(u,H)$, respectively. Fix
$H \in C^{1,2}([0,T]\times [-1,1])$. Since $\chi$ is concave and
$L_0$, $B$ are linear in $u$, $\hat{J}_H(\pi^{\epsilon})$ is less than
or equal to
\begin{equation}
\label{eq: f2}
\begin{aligned}
& \epsilon \, L_0(\rho,H) \,+\,  (1-\epsilon) \, L_0(u,H)
- \epsilon\, \int_0^T B(\rho_t,H_t) dt 
- (1-\epsilon) \, \int_0^T B(u_t,H_t) dt  \\
&+ \epsilon\int_0^T \int_{-1}^1\Bigl \{ D(u^{\epsilon}_t) \nabla
\rho_t \nabla H_t - D(u^{\epsilon}_t) \chi(\rho_t) (\nabla
H_t)^2 \Bigl \} dx\, dt \\
&+(1-\epsilon)\int_0^T \int_{-1}^1 \Bigl \{D(u^{\epsilon}_t)\nabla
u_t \nabla H_t-D(u^{\epsilon}_t) \chi(u_t) (\nabla H_t)^2 \Bigl \}
dxdt \;.
\end{aligned}
\end{equation}

We first estimate the terms involving the solution of the hydrodynamic
equation. We claim that
\begin{equation}
\label{08}
\begin{aligned}
\underset{\epsilon \to 0}{\lim}\, 
& \underset{H \in C^{1,2}([0,T]\times [-1,1])} {\mbox{ sup }}
\Big \{ \epsilon \, L_{0}(\rho,H) \,-\,
\epsilon\, \int_0^T B(\rho_t,H_t) dt \\
& \qquad \quad + \, \epsilon\int_0^T \int_{-1}^1\Big[ \, D(u^{\epsilon}_t)
\nabla \rho_t \nabla H_t - D(u^{\epsilon}_t)
\chi(\rho_t) (\nabla H_t)^2 \Big] \,dx\, dt \, \Big \} \leq  0 \;.
\end{aligned}
\end{equation}
Since $\rho$ is a weak solution of hydrodynamic equation (\ref{10}),
and $\mf c-u= \mf c (1-u)-u(1-\mf c)$, the expression inside braces is
equal to
\begin{align*}
-\epsilon\int_0^T B_0(\rho_t,H_t) dt
\,+\, \epsilon\int_0^T \int_{-1}^1 \Big\{\, \big [\, (D(u^{\epsilon}_t)
- D(\rho_t)\,\big]
\nabla \rho_t \nabla H_t - D(u^{\epsilon}_t)\chi(\rho_t)
(\nabla H_t)^2 \Bigl \}  dx\, dt \;,
\end{align*}
where
\begin{align*}
B_0(u_t,H_t) \, & =\, u_t(1) \,[1-{\mf b}]\, \Lambda (- H_t(1))
\,+\, [1-u_t(1)]\, {\mf b}\, \Lambda (H_t(1))  \\
& =\, u_t(-1) \,[1-{\mf a}]\, \Lambda (- H_t(-1))
\,+\, [1-u_t(-1)]\, {\mf a}\, \Lambda (H_t(-1)) \;,
\end{align*}
and $\Lambda (x) = e^x-x-1$. Clearly, $B_0(u_t,H_t) \geq 0$, and the
first term of the penultimate displayed equation is negative. On the
other hand, by Young's inequality the second term is bounded above by
\begin{align*}
\epsilon\, \int_0^T \int_{-1}^1
\frac{ [\, D(u^{\epsilon}_t)-D(\rho_t)\,]^2
(\nabla \rho_t)^2}{4D(u^{\epsilon}_t)\chi(\rho_t)}
\,dx\, dt \;.
\end{align*}
Since the diffusion coefficient $D$ is bounded below by a strictly
positive constant and above by a finite constant, the previous
expression is less than or equal to
$C_0 \, \epsilon \, \mathcal{Q}_{[0,T]}(\rho)$. By Corollary \ref{cor:
hydro ener fini}. the solution of the hydrodynamic equation has finite
energy, which proves claim \eqref{08}.

We turn to the terms in (\ref{eq: f2}) which are multiplied by
$1-\epsilon$. Our goal is to replace $u^\epsilon$ by $u$ to get an
expression bounded by $(1-\epsilon) I_{[0,T]}(u)$. Rewrite
$D(u^{\epsilon}_t) \chi(u_t)$ as
$D(u_t) \chi(u_t) \times [D(u^{\epsilon}_t)/D(u_t)]$.  Since
$u^{\epsilon}=u+\epsilon(\rho-u)$, and $|\rho-u| \leq 1$, adding and
subtracting in the numerator $D(u)$ yields that
$D(u^{\epsilon}_t)/D(u_t) \ge 1 - C_0 \epsilon$ for some finite
constant $C_0$ because $D$ is bounded below by a strictly positive
constant. In conclusion, since $\chi D=\sigma$, the terms in (\ref{eq:
f2}) which are multiplied by $1-\epsilon$ are bounded by
\begin{equation*}
L_0(u,H) - \int_0^T B(u_t,H_t) dt  
+ \int_0^T \int_{-1}^1 \Bigl \{D(u^{\epsilon}_t)\nabla
u_t \nabla H_t- (1-C_0\epsilon) \sigma(u_t) (\nabla H_t)^2 \Bigl \}
dxdt \;.
\end{equation*}
Fix $\eta>0$ which converges to $0$ after $\epsilon$, and add and subtract
$D(u_t)\nabla u_t \nabla H_t$. By Young's inequality, and since
$|D(u^{\epsilon}_t)-D(u_t)| \le C_0 \epsilon$,
\begin{equation*}
\int_0^T \int_{-1}^1  \{\, D(u^{\epsilon}_t) - D(u_t)\}
\nabla u_t \nabla H_t \, dx\, dt
\;\le\; \eta \int_0^T \int_{-1}^1  \sigma(u_t) (\nabla H_t)^2
\, dx\, dt \,+\,
\frac{C_0 \epsilon^2}{\eta} \mc Q_{[0,T]}(u)\;.
\end{equation*}
As $u$ belongs to $\mc F_0$, it has finite energy, and the second term
on the right-hand side vanishes as $\epsilon\to 0$. Therefore, the
next to the last displayed equation is bounded by
\begin{equation*}
L(u,H) - \int_0^T B_0(u_t,H_t) dt  
\,-\, (1-\eta-C_0\epsilon)
\, \int_0^T \int_{-1}^1  \sigma(u_t) (\nabla H_t)^2 
dx\, dt \, +\, C_0(u) \epsilon^2 \;.
\end{equation*}
Mind that we changed $L_0$, $B$ to $L$, $B_0$, respectively.

Let $a= (1-\eta-C_0\epsilon)^{-1}>1$, $H= aG$. Since $e^{ax} - ax - 1
\ge a(e^x-x-1)$ for $x\in \bb R$, the previous expression is bounded by
\begin{equation*}
a \,\Big\{\, L(u,G) - \int_0^T B_0(u_t,G_t) dt  
\,-\,\, \int_0^T \int_{-1}^1  \sigma(u_t) (\nabla G_t)^2 
dx\, dt \,\Big\} \, +\, C_0(u) \epsilon^2 \;.
\end{equation*}
This expression is bounded by $a I_{[0,T]} (u) + C_0(u) \epsilon^2
$. It remains to let $\epsilon \to 0$ and then $\eta \to 0$ to
complete the proof of the lemma.
\end{proof}

\subsection*{ $\mathcal{F}_2$ density}

Let $\mathcal{F}_2$ be the set of all paths $\pi(t,dx)=u(t,x)dx$ in
$\mathcal{F}_1$ with the property that there exist
$\delta_1,\delta_2>0$ such that $u_t=\rho_t $ in $[0,\delta_1]$ and
$u$ is constant in $[\delta_1,\delta_1+\delta_2]$.

\begin{lem}
\label{lc03}
The set $\mathcal{F}_2$ is $I_{[0,T]}(.|\gamma)$-dense.
\end{lem}

\begin{proof}
Fix $\pi(t,dx)=u(t,x)dx$ in $\mathcal{F}_1$. By Lemma \ref{lc01}, it
is enough to show that there exists a sequence
$(\pi^n)_{n \in \mathbb{N}}$ in $\mathcal{F}_2$ such that
$I_{[0,T]}(\pi^n|\gamma) \to I_{[0,T]}(\pi|\gamma)$ and
$\pi^n\to \pi$.

By definition of $\mathcal{F}_1$, there exists $\delta>0$ such that
$u_t=\rho_t$ for $t \in [0, \delta]$. Fix a sequence
$r_n\uparrow \delta$.  Since $\rho$ has a finite energy, there exists
$r_n < t^n_1 < \delta$ such that
\begin{equation}
\label{09}
\int_{-1}^1\frac{D(\rho_{t^n_1}) \, (\nabla \rho_{t^n_1})^2 }
{\sigma(\rho_{t^n_1})} \, dx
\,\le\, \frac{2}{\delta - r_n} \int_{r_n}^\delta 
\int_{-1}^1\frac{D(\rho_{t}) \, (\nabla \rho_{t})^2 }
{\sigma(\rho_{t})} \, dx \, dt
<+ \infty \;.
\end{equation}

Consider a sequence $t^n_2$ such that $t^n_2 \to 0$,
$t^n_1+ t^n_2 < \delta$.  Define the path
$\pi^{(n)}(t,dx)=u^{(n)}(t,x)dx$, $n \in \mathbb{N}^*$, by
$u_t^{(n)}=u_t$ on $[0,t^n_1]$, $u_t^{(n)}=u_{t^n_1}$ on
$[t^n_1, t^n_1+ t^n_2]$ and $u_t^{(n)}=u_{t-t^n_2}$ on
$[t^n_1+ t^n_2,T]$.  By construction, $\pi^{(n)}$ belongs to
$\mathcal{F}_1$ and $\pi^{(n)} \to \pi$.  By lower semicontinuity of
the rate function, it remains to show that
$\limsup_n I_{[0,T]}(\pi^{(n)}|\gamma) \leq I_{[0,T]}(\pi|\gamma)$.

By Corollary \ref{cor: separtion I},
\begin{align*}
I_{[0,T]}(\pi^{(n)}|\gamma)
\,=\, I_{[0,t^n_1]}(\pi^{(n)}|\gamma)
\,+\, I_{[t^n_1, s^n_2]}(\pi^{(n)}) \,+\, I_{[s^n_2,T]}(\pi^{(n)}) \;.
\end{align*}
where $s^n_2 = t^n_1+ t^n_2$.  By definition of $\pi^{(n)}$,
$I_{[0,t^n_1]}(\pi^{(n)}|\gamma)=I_{[0,t^n_1]}(\rho|\gamma)=0$ for all
$n$.

We claim that
$\underset{n \to + \infty}{\lim} I_{[t^n_1,s^n_2]}(\pi^{(n)}) = 0 $.
Fix $H\in C^{1,2}([t^n_1,s^n_2] \times [-1,1])$. Since $\pi^{(n)}$ is
constant on $[t^n_1,s^n_2]$, $L_0(H)=0$, and by Young's inequality,
\begin{align*}
\hat{J}_{[t^n_1,s^n_2],H}(\pi^{(n)}) \,\le\, 
(s^n_2 - t^n_1)\,
\int_{-1}^1 \frac{D(\rho(t^n_1,x)) [\nabla \rho(t^n_1,x)]^2}
{4\, \sigma(\rho (t^n_1, x)} \, dx  
\,-\, \int_{t^n_1}^{s^n_2}B(\rho_{t^n_1},H_t)dt \;.
\end{align*}
By \eqref{09}, the first term is bounded by
\begin{equation*}
\int_{r_n}^{\delta} \int_{-1}^1 \frac{D(\rho(t,x)) [\nabla \rho(t,x)]^2}
{2\, \sigma(\rho (t, x)} \, dx \, dt\;, 
\end{equation*}
which vanishes as $n\to\infty$ because $\rho$ has finite energy. On
the other hand, as $-\, B(a,H) \le 4$, the second term is bounded by
$4(s^n_2 - t^n_1) \to 0$. This proves that
$I_{[t^n_1, s^n_2]}(\pi^{(n)}) \to 0 $.

Finally, by construction,
$I_{[s^n_2, T]}(\pi^{(n)}) =  I_{[t^n_1, T - t^n_2]}(\pi) \le
 I_{[t^n_1, T]}(\pi)$, which completes the proof of the lemma.
\end{proof}

\subsection*{ $\mathcal{F}_3$ density}

Let $\mathcal{F}_3$ be the set of all paths $\pi(t,dx)=u(t,x)dx$ in
$\mathcal{F}_2$ for which there exist $\delta, \delta_1>0$ such that
$u_t=\rho_t$ in $[0,\delta]$, $u_t$ is constant in
$[\delta,\delta+\delta_1]$, and $u(.,x)$ belongs to
$C^{\infty}([\delta,T] )$ for each $x \in [-1,1]$.

The proof of the $\mathcal{F}_3$-density relies one the representation
of the rate functional derived in Proposition \ref{p02}.  Fix
$\pi \in \mathcal{F}_2$. By definition, $\pi(t,dx)=u(t,x)dx$ belongs
to $D([0,T], \mathcal{M})$, has finite energy, and there exist
$(\delta,\delta_1) \in [0,T]^2$ such that:
\begin{align*}
u_t=
\left\{ 
\begin{array}{ccc}
\rho_t &   0 \leq t \leq \delta\;, \\
u_{\delta}  &  \delta \leq t \leq \delta + \delta_1\;, 
\end{array}
\right.
\end{align*}
where $\rho$ is the weak solution of the hydrodynamic
equation. Moreover, for every $\delta_2$, there exists $\epsilon>0$
such that $\epsilon \leq u(t,x) \leq 1-\epsilon$ for all
$(t,x) \in [\delta_2,T] \times [-1,1]$.

Fix a smooth function $\varphi$ such that
supp$(\varphi) \subset\, ]0,1[$ and $\int_0^1\varphi(t)dt=1$.  Let
$\varphi_{\epsilon}(t)= (1/\epsilon)\, \varphi(t/\epsilon)$,
$\epsilon>0$. We extend the definition of $u $ to $[0,2T]$ setting
$u(T+t,x)=u(T-t,x)$ for $0\le t\le T$. For $n$ such that
$1/n<\delta_1$, let
\begin{align}
\label{eq : def u^n F_3}
u^n(t)= \left \{
\begin{array}{ccc}
u(t) &  0 \leq t \leq \delta \\
(u*\varphi_{\frac{1}{n}}) (t)  & \delta \leq t \leq T \;.
\end{array}
\right .
\end{align}
Note that $u(t) = u(\delta)$ for
$\delta \le t\le \delta+ \delta_1 - (1/n)$.  Set
$\pi^n(t,dx)=u^n(t,x)dx$.

\begin{lem}
\label{l06}
Fix $\pi \in \mathcal{F}_2$.  Each sequence in $\bb N$ has a
subsequence $(\psi(n))_{n \in \mathbb{N}}$ such that
\begin{align*}
\underset{n \to \infty}{\lim}
I_{[0,T]}^{(1)}(\pi^{\psi(n)}|\gamma) = I_{[0,T]}^{(1)}(\pi|\gamma)    
\end{align*}
\end{lem}

\begin{proof}
By Proposition \ref{p02}:
\begin{align*}
I_{[0,T]}^{(1)}(\pi^n|\gamma) &=\frac{1}{4}\int_0^T \Bigl
\{\<M_n(t),M_n(t)\>_{\sigma(u_t^n)^{-1}}-
\<M_n(t)\sigma(u_t^n)^{-1}\>^2\<\sigma(u_t^n)^{-1}\>^{-1} \Bigl \} dt
\end{align*}
\\where $M_n=P_n+D(u^n)\nabla u^n$ and $P_n \in L^2(\sigma(u^n)^{-1})$,
introduced in Lemma \ref{l01}, is such that 
\begin{align*}
-\int_0^T\<\pi^n,\partial_t H_t\>dt=\<\<\nabla H, P_n\>\> 
\end{align*}
for every $H \in C_K^{\infty}(]0,T[\times ]-1,1[)$. An elementary
computation shows that
\begin{align*}
P_n(t)=\left \{ 
\begin{array}{ccc}
P(t)    & 0 \le t \leq \delta  \\
(P*\varphi_{\frac{1}{n}})(t)   & \delta \le t \le T\;,
\end{array}
\right .    
\end{align*}
provided we extend the definition of $P$ to $[0,2T]$ setting
$P(T+t)=-P(T-t)$ for $t \in [0,T]$.

Since $u^n=u$ and $P_n=P$ on the time-interval $[0,\delta]$, 
in view of the formula for $I_{[0,T]}^{(1)}(\pi^n|\gamma)$, we only
need to examine the interval $[\delta, T]$. In this interval,
\begin{align*}
u^n=u*\varphi_{\frac{1}{n}} \;, \quad
\nabla u^n = \nabla u * \varphi_{\frac{1}{n}}
\;,\;\; \text{ and }\;\; P_n=P*\varphi_{\frac{1}{n}}    \;.
\end{align*}

As $\pi$ belongs to $\mc F_2$, there exists $\epsilon>0$ such that
$\epsilon \le u \le 1-\epsilon$ on the interval $[\delta, T]$.
In particular, on the interval $[\delta , T]$, the space 
$L^2$ and $L^2(\sigma(u^n)^{-1})$ coincide.

Fix the set $[\delta, T]$.  As $u^n$ converge to $u$ in $L^1$, there
exists a subsequence, still denoted by $u^n$, which converges a.s. to
$u$. Hence, the uniformly bounded sequences $D(u_n)$,
$\sigma (u_n)^{-1}$ converge a.s. to $D(u)$, $\sigma (u)^{-1}$,
respectively.

Since $P$, $\nabla u$ belong to $L^2$, $D$ is bounded and $u_n$
converge a.s. to $u$, $P_n$, $D(u_n)\, \nabla u_n$
converge in $L^2$ to $P$, $D(u)\, \nabla u $, respectively.

Putting together the previous estimates yields that there exists a
subsequence $\psi(n)$ such that
\begin{align*}
&\underset{n \to \infty}{\lim}
\frac{1}{4}\int_{\delta}^T \Bigl
\{ \, \Vert\, M_{\psi(n)}(t)\,\Vert^2 _{\sigma(u_t^{\psi(n)})^{-1}}
- \<M_{\psi(n)}(t)\, \sigma(u_t^{\psi(n)})^{-1}\>^2
\<\sigma(u_t^{\psi(n)})^{-1}\>^{-1} \Bigl \} \, dt \\
&\quad =\, \frac{1}{4}\int_{\delta}^T \Bigl
\{ \Vert\, M_{\psi(n)}(t)\,\Vert^2_{\sigma(u_t)^{-1}}
- \<M(t)\sigma(u_t)^{-1}\>^2
\<\sigma(u_t)^{-1}\>^{-1} \Bigl \} \, dt\;,
\end{align*} 
which completes the proof of the lemma.
\end{proof}

\begin{lem}
\label{lem: aux2 densite F3} 
Fix $\pi \in \mathcal{F}_2$.  Each subsequence of $\bb N$ has a
subsequence $(\psi(n))_{n \in \mathbb{N}}$ such that
\begin{align*}
\underset{n \to \infty}{\limsup} I_{[0,T]}^{(2)}(\pi^{\psi(n)}|\gamma)
\leq I_{[0,T]}^{(2)}(\pi|\gamma) \;.
\end{align*}
\end{lem}

\begin{proof}
Recall from Proposition \ref{p03} the formula for the functional
$I_{[0,T]}^{(2)}(\,\cdot\,|\gamma)$. As in the proof of the previous
lemma, we may concentrate on the interval $[\delta, T]$.

Since $u$ belongs to $\mc F_2$, there exists $\epsilon>0$ such that
$\epsilon \le u \le 1- \epsilon$ on the time-interval $[\delta,T]$.
\begin{equation*}
\begin{aligned}
& \int_\delta^T \Phi_t\bigl(u^n_t(-1),u^n_t(1),g^n(t),h^n(t)\bigl) \,
dt \\
&\quad =\;
\int_\delta^T \sup_{\alpha, \beta}
\Bigl \{ \alpha g^n(t) + \beta h^n(t) -
\Upsilon_t(u^n_t(-1),u^n_t(1) , \alpha,\beta) \Bigl \} \, dt\;,
\end{aligned}
\end{equation*}
where $g_n(t)$, $h_n(t)$ are given by \eqref{04} with $u^n$ replacing
$u$. Mind that this substitution affects the definitions of $P^X_n$,
$P^{\bb 1}_n$, $M^n$ and $\mc S^n$.  By the bound on $\alpha$, $\beta$
presented in Proposition \ref{p03}, we may restrict the supremum to
pairs $(\alpha,\beta)$ such that
$e^{|\alpha|} + e^{|\beta|} \leq C_0 (1 + |g_n(t)| + |h_n(t)|)$.

Most of the terms appearing in the formula for $g^n$ and $h^n$ are
convolutions with $\varphi_{\frac{1}{n}}$. For example,
$P_n^X(t)=\bigl(P^X*\varphi_{\frac{1}{n} }\bigl)(t)$. We wish to
replace all terms which appear in the formula for $g^n$, $h^n$,
$\mc S^n$ and are not expressed as convolutions by convolutions. We
present the details for one of them. The others are handled similarly.

Consider the term $(\alpha-\beta)\<D(u_n)\nabla u_n\>$. Add and
subtract $(\alpha-\beta)\<D(u)\nabla u\>*\varphi_{\frac{1}{n}}$. Since
the supremum of a sum is bounded by the sum of the supremums, to carry
out this replacement, it is enough to show that
\begin{align*}
\underset{n \to + \infty}{\lim}
\int_{\delta}^T\underset{\alpha, \beta}
{\mbox { sup }} \Bigl \{(\alpha-\beta)
\bigl [ \<D(u_n)\nabla u_n\>(t)
- \<D(u)\nabla u\>*\varphi_{\frac{1}{n}}(t)\bigl ]  \Bigl \}
\, dt =0\;, 
\end{align*}
where, recall, the supremum is restricted to pairs whose exponentials
are bounded. By Young's inequality $2xy \leq A^{-1} x^2+Ay^2$, and
since $a^2 \leq C_0 + e^{|a|}$ for some finite constant $C_0$ whose
value may change from line to line, integral is bounded by:
\begin{align*}
\frac{C_0}{A}\int_{\delta}^T \bigl ( 1+|g_n(t)|+|h_n(t)|\bigl )  \, dt
\,+\, A\int_{\delta}^T \bigl [ \<D(u_n)\nabla u_n\>(t)
- \<D(u)\nabla u\>*\varphi_{\frac{1}{n}}(t)\bigl ]^2  \, dt
\end{align*}
for all $A>0$.

Since $D(.)$ is continuous and $u_n$, $\nabla u_n$ converge in
$L^2([\delta, T]\times [-1,1])$ to $u$, $\nabla u$, respectively, the
second term vanisheds as $n \to + \infty$. For a similar reasons, the
first term vanishes, as $n \to + \infty$, and then $A\to\infty$.

After all replacements, we obtain that
\begin{equation*}
\begin{aligned}
& \underset{n \to + \infty}{\limsup}\,
I_{[\delta , T]}^{(2)}(\pi^n) 
\leq\underset{n \to + \infty}{\limsup} \,
\frac{1}{4}\int_{\delta}^T \\
& \sup_{\alpha, \beta} \Big\{ \int_0^{1/n} \varphi_n(s)\,
\Big \{ \, \alpha \, g(t+s) + \beta \, h(t+s) 
\, -\, \Upsilon_{t+s} (u_{t+s}(-1),u_{t+s}(1),\alpha ,\beta ) \Big
\} \, ds\, \Big\}\, dt \;.
\end{aligned}
\end{equation*}
By moving the supremum inside the integral, the right-hand side is
seen to be bounded by
\begin{align*}
\underset{n \to + \infty}{\limsup}\,
\frac{1}{4}\int_{\delta}^T
[\Phi_\cdot (u_{\cdot} (-1), u_{\cdot} (1),
g(\cdot), h(\cdot)) * \varphi_n] (t)\, dt \;.
\end{align*}
Since $\Phi$ is integrable, this expression is equal to
\begin{align*}
\frac{1}{4}\int_{\delta}^T \Phi_t(u_t(1),u_t(-1),g(t),h(t))
\, dt\;,
\end{align*}
which completes the proof of the lemma.
\end{proof}

\begin{lem}
The set $\mathcal{F}_3$ is $I_{[0,T]}(.|\gamma)$-dense.
\end{lem}

\begin{proof}
Fix $\pi(t,dx)=u(t,x)dx$ in $\mathcal{F}_2$ such that
$I_{[0,T]}(\pi|\gamma)<+\infty$. Consider
$(\pi^n)_{n \in \mathbb{N}^*}$ as in (\ref{eq : def u^n F_3}). It's
clear that $\pi^n$ belongs to $\mathcal{F}_3$ and converges to
$\pi$. Since, by Theorem \ref{thm: semi cont de I}, the rate function
is lower semicontinuous, to complete the proof it remains to show that
$\limsup_n I_{[0,T]}(\pi^n|\gamma) \leq I_{[0,T]}(\pi|\gamma) $. This
follows from the two previous lemmata.
\end{proof}

\subsection*{ $\Pi_\gamma$ density}

Recall the definition of the set $\Pi_\gamma$ introduced at the
beginning of this section.  Denote by $(P_t^{(D)})_{t \geq 0}$ and
$(P_t^{(N)})_{t \geq 0}$ the semigroup associated to the Laplacian on
$[-1,1]$ with Dirichlet, Neumann boundary conditions,
respectively. The proof of the next result is similar to the one of
Lemmata 5.6 and 5.7 in \cite{fgln}.  It relies on the decomposition of
the rate function presented in Proposition \ref{p04} and uses the fact
that $D$ is bounded.

\begin{lem}
Fix $\pi(t,dx)=u(t,x)dx$ in $\mathcal{F}_3$ such that
$I_{[0,T]}(\pi|\gamma)<\infty$. Let $\delta$, $\delta_1>0$ be the
positive constants such that $u_t=\rho_t$ on $[0,\delta ]$ and
$u_t=\rho_{\delta}$ on $[\delta,\delta+\delta_1 ]$.  Let
$\kappa\colon [0,T] \to [0,1]$ be a nondecreasing smooth function such
that
\begin{align*}
\kappa(t)=
\left \{ 
\begin{array}{ccc}
0 & 0 \leq t \leq \delta  \;, \\
0 < \kappa (t) < 1 & \delta  < t < \delta +\delta_1 \;, \\
1 & \delta+\delta_1   \leq t \leq T \;.
\end{array}
\right . \;.
\end{align*}
Let $\kappa_n(t)= (1/n) \, \kappa(t)$, $n\ge 1$, and set
$\pi^n(t,dx)=u^n(t,x)\, dx$ where
$u_t^n=w_t+P^{(D)}_{\kappa_n(t)}[u_t-w_t]$ and
\begin{equation*}
w_t(x) \;=\;  
\frac {u_t(1)+u_t(-1)} {2} \,+\, \frac{u_t(1)-u_t(-1)}{2} \, x \;.
\end{equation*}
Then, for each $n \geq 1$, $\pi^n(t,dx)=u^n(t,x)dx$ belongs to
$\Pi_\gamma$ and
$\limsup_n I_{[0,T]} (\pi^n|\gamma) \leq I_{[0,T]} (\pi|\gamma )$.
\end{lem}

\begin{proof}[Proof of Proposition \ref{p05}]
Since $\pi^n$ converges to $\pi$, the assertion of the proposition
follows from the previous lemma.
\end{proof}

We conclude this section providing an explicit formula for the rate
functions of trajectories in $\Pi_\gamma$. The proof does not require
the density profile $\gamma$ to satisfy the hypotheses of Theorem
\ref{t02}. 



\begin{proof}[Proof of Proposition \ref{l09b}]
By Corollary \ref{cor: separtion I} and Corollary \ref{cor: I hydro
nul}, $I_{[0,T]}(\pi|\gamma)=I_{[\delta ,T]}(\pi)$. The rest of the
argument is straightforward from the variational formula for
$I_{[\delta ,T]}(\pi)$. We refer to the proof of Proposition 2.6 on
\cite{fgln}.
\end{proof}

\section{The large deviations principle}
\label{sec7}

The proof of the large deviations principle follows the one presented
in \cite{blm, flm}. It differes slightly from the original one
\cite{kl} as the functional $J_H(\cdot)$ is set to be $+\infty$ on
paths with infinite energy.

The proof of the lower and upper bounds do not require
$\gamma$ to satisfy the regularity conditions stated in Theorem
\ref{t02}. It is only at the end of the proof of the lower bound,
where we use the $I_{[0,T]}$-density of the set $\Pi_\gamma$, that
these hypotheses are needed. See Remark \ref{rm3}. 

The proof of the upper bound is identical to the one for
one-dimensional symmetric exclusion processes in weak contact with
reservoirs \cite{fgln} which relies on results presented in
\cite{flm}. Full details are given in Section 3.4 in \cite{these}.

\subsection*{Lower bound}

The proof of the lower bound relies on the hydrodynamic limit of
weakly asymmetric dynamics and on the
$I_{[0,T]}(\,\cdot\,|\gamma)$-density of the set $\Pi_\gamma$
introduced in the previous section.

Fix $0<t_0<T$, a profile $\gamma\colon [-1,1] \to [0,1]$, and a
function $H$ in $C([0,T] \times [-1,1])$ such that $H=0$ in
$[0,t_0] \times [-1,1]$ and $H \in C^{1,2}([t_0,T] \times [-1,1])$.
The concept of weak solutions, introduced in Definition \ref{d01} for
the hydrodynamic equation, can be extended to the equation
\begin{equation}
\label{c03}
\left\{
\begin{aligned}
& \partial_t u_t  - \nabla[D(u_t)\, \nabla u_t]
= 2 \nabla [ \sigma(u_t)\nabla H_t]
\quad\text{in}\;\; [0,T]\times (-1,1)\;,  \\
& D(u_t) \nabla u_t - 2 \sigma(u_t) \nabla H_t
= -u_t \, (1-{\mf b}) \, e^{-H_t}+{\mf b}\, [1-u_t]\,
e^{H_t} \quad\text{in}\;\; [0,T] \times \{1\} \;, \\
& D(u_t) \nabla u_t - 2 \sigma(u_t) \nabla H_t
= u_t \, (1-{\mf a}) \, e^{-H_t} - {\mf a}\, [1-u_t]\,
e^{H_t} \quad\text{in}\;\; [0,T] \times \{-1\} \;,
\end{aligned}
\right.
\end{equation}
with initial condition $\gamma$.
Corollary \ref{t03} states that there exists a unique weak solution.

Fix a function $G$ in $C([0,T] \times [-1,1])$, and consider the
generator $L_N^G$ given by $L^G_N = L_{N,0}^G + L_{N,b}^G$, where
\begin{equation*}
(L_{N,0}^Gf)(\eta)=  N^2  \sum_{x \in \Lambda_N}
r_{x,x+1}(\eta)\, e^{- [\eta(x+1)-\eta(x)]\, [G_t(\frac{x+1}{N})-G_t(\frac{x}{N})]}
[\, f(\eta^{x,x+1})-f(\eta)\,]\;, 
\end{equation*}
and the sum is carried over all $x\in \bb Z$ such that
$\{x , x+1\} \subset \Lambda_N$. On the other hand,
\begin{equation*}
\begin{aligned}
(L_{N,b}^Gf)(\eta) \, & =\, N\, r_L(\eta) \,
e^{-\, [1\, -2\eta(-N+1)\,]\, G_t(-1+\frac{1}{N})}\, [\,
f(\eta^{-N+1})-f(\eta)\,] \\
& +\, N\, r_R(\eta) \,
e^{-\, [1\, -2\eta(N-1)\,]\, G_t(1-\frac{1}{N})}\, [\, f(\eta^{N-1})-f(\eta)\,]\;,
\end{aligned}
\end{equation*}
These generators are time-dependent, though this dependence does not
appear in the notation.

For a probability measure $\mu$ on $\mf S_N$, denote by
$\mathbb{P}_{\mu}^{G,N}$ the probability measure on
$D([0,T], \mf S_N)$ induced by the Markov chain whose generator is
$L^G_N$ starting from $\mu$. Let
$\mathbb{Q}_{\mu}^{G,N} = \mathbb{P}_{\mu}^{G,N} (\Pi^N)^{-1}$.  The
next result provides a law of large numbers for the empirical measure
for weakly asymmetric gradient exclusion processes in mild contact
with reservoirs.

\begin{proposition}
\label{pf01}
Fix $0<t_0<T$, a profile $\gamma\colon [-1,1] \to [0,1]$, and a
function $H$ in $C([0,T] \times [-1,1])$ such that $H=0$ in
$[0,t_0] \times [-1,1]$ and $H \in C^{1,2}([t_0,T] \times
[-1,1])$. Let $(\mu_N:N\ge 1)$ be a sequence of probability measures
on $\mf S_N$ associated to $\gamma$.  Then, the sequence of
probability measures $(\mathbb{Q}_{\mu_N}^{H,N} : N \in \mathbb{N})$
converges weakly to the probability measure $\mathbb{Q}$ concentred on
the trajectory $\pi(t,dx)=u(t,x)dx$ where $u$ is the unique weak
solution of equation \eqref{c03} with initial condition $\gamma$.
\end{proposition}

The proof of this result is similar to the one presented in Section
\ref{sec3}, and left to the reader.  For two probability measures
$\mathbb{P}$, $\mathbb{Q}$ on $D([0,T], \mf S_N)$, denote by
$\mathcal{H}_N(\mathbb{P}|\mathbb{Q})$ the relative entropy of $\bb P$
with respect to $\bb Q$:
\begin{equation*}
\mathcal{H}_N(\mathbb{P}|\mathbb{Q}) = \int \log \frac{d\bb P}{d\bb Q}\,
d\bb P\;.
\end{equation*}

Next result follows from the previous proposition and Proposition
\ref{l09b}.

\begin{lem}
\label{lf01}
Fix a density profile $\gamma\colon [-1,1] \to [0,1]$, and let
$(\eta^N:N\ge 1)$ be a sequence of configurations on $\mf S_N$
associated to $\gamma$. Under the hypotheses of the previous
proposition on $H$,
$\underset{N \to \infty}{\lim} \frac{1}{N}
\mathcal{H}(\mathbb{P}_{\eta^N}^{H,N}|\mathbb{P}_{\eta^N})=I_{[0,T]}(u|\gamma)$
where $u(t,x)$ is the unique weak solution of equation (\ref{c03}).
\end{lem}

Fix a density profile $\gamma\colon [-1,1]\to [0,1]$. Let
$\pi(t,dx) = u(t,x) dx$ be a trajectory in $\Pi_\gamma$. Recall from
the end of the previous section that for each $0<t\le T$, there exists
a unique solution, denoted by $H_t$, of the elliptic equation
\eqref{5-01b}.  Moreover the solution satisfies the hypotheses of
Proposition \ref{pf01} and Lemma \ref{lf01}.


Let $(\eta^N:N\ge 1)$ be a sequence of configurations on $\mf S_N$
associated to $\gamma$. The proof of Lemma 10.5.4 in \cite{kl} yields
that for every open set $\mc O \subset D([0,T], \mathcal{M})$,
\begin{align*}
\underset{N \to + \infty}{\mbox{ lim inf }} \frac{1}{N}
\log \mathbb{Q}_{\eta^N}(\mc O) \geq -\underset{\pi}
{\mbox{ inf }} I_{[0,T]}(\pi|\gamma) \;,
\end{align*}
where the infimum is carried over all trajectories $\pi(t,dx)$ in
$\Pi_\gamma \cap \mc O$. To complete the proof of the lower bound it
remains to recall that the set $\Pi_\gamma$ is
$I_{[0,T]}(\,\cdot\,|\gamma)$-dense. At this point, and only here, we
need $\gamma$ to fulfil the hypotheses of Theorem \ref{t02}.

\section{Uniqueness of weak solutions}
\label{sec8}

In this section, we prove uniqueness of weak solutions of the
hydrodynamic equation for weakly asymmetric dynamics. As many results
are standard, they are just stated, and we refer to Section 2.4 in \cite{these} for
detailed proofs.

For $\delta>0$ small, denote by $\varphi_{\delta}$ the function defined by
\begin{align*}
\varphi_{\delta}(\cdot)  \,:=\,  \frac{x^2}{2\delta}\, \chi_{[-\delta,
\delta]} (\cdot) \,+\, [\, |x|-(\delta/2)\,]
\chi_{[-\delta, \delta]^c} (\cdot) \;.
\end{align*}
Fix $H \in C([0,T]\times [-1,1])$, for which there exists $0<t_H<T$
such that $H=0$ on $[0,t_H]\times [-1,1]$ and
$H \in C^{1,2}([t_H,T]\times [-1,1])$.  The proofs of next results are
similar to the ones of \cite[Section 7]{flm}. Details can be found in Section 3.5 in 
\cite{these}.

\begin{lem}
\label{lem: varphi_detla(u)}
Fix two initial profiles $\rho_i \colon \Omega \to [0,1]$, $i=1$,
$2$. Let $u^{(i)}$ be a weak solution of the hydrodynamic equation
(\ref{c03}) with initial condition $\rho_i$. Then, for all
$0< t \leq T$,
\begin{align*}
\int_{-1}^1 \varphi_{\delta}(w_t)\, dx
- \int_{-1}^1 \varphi_{\delta}(w_0)\, dx
\, =\, &-\, \int_0^t [\, {\mf b}\, e^{H_s(1)}+(1-{\mf b})\,
e^{-H_s(1)}\, ]
\, w_s(1) \, \varphi_{\delta}'(w_s(1))\, ds \\
\,& -\, \int_0^t [\, {\mf a}\, e^{H_s(-1)}
+ (1-{\mf a}) \, e^{-H_s(-1)}]\,
w_s(-1)\, \varphi_{\delta}'(w_s(-1)) \, ds \\
& -\, \int_0^t \int_{-1}^1 \big\{\, D(u^{(2)}_s) \, \nabla u^{(2)}_s
- D(u^{(1)}_s) \, \nabla u^{(1)}_s\,\big\}
\nabla w_s \, \varphi_{\delta}''(w_s ) \, dx \, ds \\
& + \, 2 \, \int_0^t \int_{-1}^1 \big\{\, \sigma (u^{(2)}_s)
- \sigma(u^{(1)}_s) \, \big\} \, \nabla H_s\,
\nabla w_s \, \varphi_{\delta}''(w_s ) \, dx \, ds \;, 
\end{align*}
where $w=u^{(2)}-u^{(1)}$.
\end{lem}

\begin{cor}
\label{t03}
Fix two initial profiles $\rho_i \colon \Omega \to [0,1]$, $i=1$,
$2$. Let $u^{(i)}$ be a weak solution of the hydrodynamic equation
(\ref{c03}) with initial condition $\rho_i$. Then,
$t\mapsto \| u^{(2)}_t-u^{(1)}_t \|_1$ is non-increasing.  In
particular, there is at most one weak solution of the parabolic
equations (\ref{10}) and \eqref{c03}.  Moreover, if $H=0$, then
\begin{equation*}
\int_0^{\infty} \| u^{(2)}_t-u^{(1)}_t \|_1^2 \,<\,  \infty \;.
\end{equation*}
\end{cor}

Denote by $\color{blue} C^{1+ \beta/2, 2+\beta}([0,T]\times [-1,1])$
the functions $H$ in $C^{1,2}([0,T]\times [-1,1])$ whose derivatives
$\partial^2_x H$, $\partial_t H$ are H\"older continuous in the $x$
variable with parameter $\beta$ and H\"older continuous in the $t$
variable with parameter $\beta/2$.

\begin{proposition}
\label{p07}
Fix an initial profile $\gamma\colon \Omega \to [0,1]$ in
$C^{2+\beta}([-1,1])$ for some $0<\beta<1$ and satisfying the boundary
conditions \eqref{fv-01}.  Let $u$ be the weak solution of the
hydrodynamic equation (\ref{c03}) with initial condition
$\gamma$. Then,

\begin{enumerate}
    
\item $u \in C^{1+\beta/2,2+\beta}([0,T] \times [-1,1])$

\item $0< u (t,x) <1$ for every $(t,x) \in ]0,T] \times [-1,1]$

\end{enumerate}

\noindent
In particular, for each $\delta >0$, there exists $\epsilon >0$ such
that $\epsilon \le u (t,x) \le 1-\epsilon$ for every
$(t,x) \in [\delta,T] \times [-1,1]$.
\end{proposition}

\begin{proof}
The proof is based on \cite[Theorem V.7.4]{Ld}.  Fix an initial
profile $\gamma\colon \Omega \to [0,1]$ satisfying the hypothesis of
the proposition, and let $u$ be the weak solution of the hydrodynamic
equation (\ref{c03}) with initial condition $\gamma$.

Let $v= u-\gamma$. Then, $v$ satisfies equations (7.1), (7.2) of
Chapter V in \cite{Ld} for some coefficients $a(x,v)$,
$b(x,v,\nabla v)$ and $\psi (x,v)$, and initial condition
$v(0,\cdot) =0$. It can be checked that all the hypotheses of
\cite[Theorem V.7.4]{Ld} are fulfilled, but the second condition in
equation (7.34), which is violated for large negative values of $v$.

Modify $b$ for this condition to hold when $v$ takes large negative
values, say when $v<-2$. With this modifications all hypotheses of
\cite[Theorem V.7.4]{Ld} are fulfilled. According to \cite[Theorem
V.7.4]{Ld}, there exists a solution in $C^{1,2}(\bb R \times [-1,1])$
to (7.1), (7.2) with the modified $b$ and initial condition
$v(0,\cdot)=0$.

Let $t_0$ be the first time that $v$ attains $-2$. We claim that
$t_0=\infty$. Assume, by contradiction that this is not the case.  As
the solution starts from $v=0$ and $v$ is continuous, $t_0>0$.  In the
time interval $[0,t_0]$, $v >-2$. Hence, in this interval the modified
$b$ coincides with the original $b$. In particular, in the time
interval $[0,t_0]$, $u=v+\gamma$ is the solution of the hydrodynamic
equation with initial condition $\gamma$. Since $u$ is positive,
$v = u-\gamma>-\gamma >-3/2$ on $[0,t_0]$, in contradiction to the
fact that $v(t_0,x) = -2$ for some $x\in[-1,1]$.

Hence $v$ is a solution in $\bb R \times [-1,1]$ to the equation
(7.1), (7.2) with initial condition $v(0,\cdot)=0$ for the unmodified
coefficient $b$, and $u=v+\gamma$ is the solution of the hydrodynamic
equation with initial condition $\gamma$ on the same region. As
$v\in C^{1,2}(\bb R \times [-1,1])$ and $\gamma\in C^{2}([-1,1])$,
$u \in C^{1,2}(\bb R \times [-1,1])$, as claimed.

We claim that $u(t,\cdot) \in ]0,1[$ at the boundary for all
$t>0$. Indeed, if $u(t,-1)= 1$, say, then, by the boundary condition
$\nabla u (t,-1)>0$ so that $u(t,x)>1$ for some $x\in [-1,1[$, in
contradiction with the fact that $u$ takes value in the interval
$[0,1]$.  To complete the proof of (2), it remains to recall the
maximum principle Theorem 3.3.5 \cite{pw}.
\end{proof}

\begin{lem}
\label{lem: regularite H}
Fix a density profile $\gamma\colon [-1,1]\to [0,1]$ and $\pi$ in
$\Pi_\gamma$. For all $0 < t \leq T$, there exists a unique
solution $H_t$ in $C^2([-1,1])$ of the elliptic equation
\eqref{5-01b}. Moreover, $H_t=0$ for $0\le t\le \delta$, where
$\delta>0$ is the constant appearing in item (a) of the definition of
the set $\Pi_\gamma$, and $H \in C^{\infty}([\delta,T]\times [-1,1])$.
\end{lem}

\begin{proof}
The proof is straightforward. As $u$ follows the hydrodynamic equation
in the interval $]0,\delta]$, $H_t=0$ is a solution of \eqref{5-01b}
for $t$ in this interval. On the other hand, for $\delta\le t\le T$,
as $u$ is smooth and bounded away from $0$ and $1$, one can derive an
explicit expression for the solution $H_t$, and check that it is
smooth in space and time.

For uniqueness, consider two solutions $H^{(i)}_t$, $i=1$,
$2$. Subtract the bulk equations, multiply by the difference
$H^{(2)}_t - H^{(1)}_t$, integrate by parts, and replace at
the boundary the gradients of $H^{(i)}_t$ using the boundary
conditions. These computations yield that the sum of three positive
terms is equal to $0$, proving uniqueness. A detailed proof can be
found in Lemma 3.6.1 in \cite{these}.
\end{proof}

\end{document}